\documentclass[11pt]{amsart}
\usepackage{geometry}
\usepackage{graphicx}
\usepackage{amssymb}
\usepackage{epstopdf}
\usepackage{amsmath,amscd}
\usepackage{url,verbatim}
\usepackage{mathtools}
\usepackage{enumitem}

\RequirePackage[colorlinks,citecolor=blue,urlcolor=blue]{hyperref}
\usepackage{breakurl}

\theoremstyle{plain}
\newtheorem{theorem}{Theorem}
\newtheorem{definition}[theorem]{Definition}
\newtheorem{lemma}[theorem]{Lemma}
\newtheorem{proposition}[theorem]{Proposition}

\newtheorem{claim}[theorem]{Claim}

\newcounter{mycount}
\newcounter{mycount2}
\newenvironment{romlist}{\begin{list}{\rm(\roman{mycount2})}%
   {\usecounter{mycount2}\labelwidth=1cm\itemsep 0pt}}{\end{list}}
\newenvironment{numlist}{\begin{list}{\arabic{mycount}.}%
   {\usecounter{mycount}\labelwidth=1cm\itemsep 0pt}}{\end{list}}

\newenvironment{Alist}{\begin{list}{\MakeUppercase{\alph{mycount}}.}%
   {\usecounter{mycount}\labelwidth=1cm\itemsep 0pt}}{\end{list}}
\newenvironment{problist}{\begin{list}{{\alph{mycount}}}%
   {\usecounter{mycount}\leftmargin=0pt\itemindent=1cm\labelwidth=1cm\itemsep 0pt}}{\end{list}}
\newenvironment{Romlist}{\begin{list}{\rm(\Roman{mycount2})}%
   {\usecounter{mycount2}\labelwidth=1cm\itemsep 0pt}}{\end{list}}
   
\numberwithin{equation}{section}
\numberwithin{theorem}{section}
\numberwithin{figure}{section}

\newcommand\HH{{\mathbb H}}
\newcommand\RR{{\mathbb R}}
\newcommand\PP{{\mathbb P}}
\newcommand\qq{\qquad}

\newcommand\De{\Delta}

\newcommand\NN{{\mathbb N}}
\newcommand\ZZ{{\mathbb Z}}

\newcommand\La{\Lambda}

\newcommand\ot{1-2\ }
\newcommand\resp{respectively}

\renewcommand\Psi{\La}

\newcommand\df{\textbf}

\title{Mixing time of Markov chains for the 1-2 model}
\author{Zhongyang Li}

\address{Department of Mathematics,
University of Connecticut,
Storrs, Connecticut 06269-1009, USA} \email{zhongyang.li@uconn.edu}
\urladdr{\url{https://mathzhongyangli.wordpress.com}}

\begin{document}
\maketitle

\begin{abstract}
A 1-2 model configuration is a subset of edges of a hexagonal lattice satisfying the constraint that each vertex is incident to 1 or 2 edges. We introduce Markov chains to sample the 1-2 model configurations on 2D hexagonal lattice and prove that the mixing time of these chains is polynomial in the sizes of the graphs for a large class of probability measures.
\end{abstract}

\section{Introduction}

In computer science and statistical physics, Markov chains play an important role in sampling and approximating counting algorithms. Usually the goal was to estimate the rate of convergence to the stationary distribution. In the past few decades, mathematicians, physicists and computer scientists prescribed certain target distance to the stationary distribution; the number of steps required to reach this target is called the \textbf{mixing time} of the chain. Deep connections were found between rapid mixing and spatial properties of spin systems. See \cite{NR98,AF99,Hag02,Je03,MT06,LPW} for more information on Markov chains and mixing.

In this paper, we study Markov chains on a statistical mechanical system called the 1-2 model.
Let $\HH=(V,E)$ be the  2D hexagonal lattice. A 1-2 model configuration $\omega=(V,E_{\omega})$ on $\HH$ is a subgraph of $\HH$, such that each vertex $v\in V$ is incident to one or two edges in $E_{\omega}$. See Figure \ref{fig:30} for an example of an 1-2 model configuration on the hexagonal lattice.

\begin{figure}[htbp]
\includegraphics[width=1.1\textwidth]{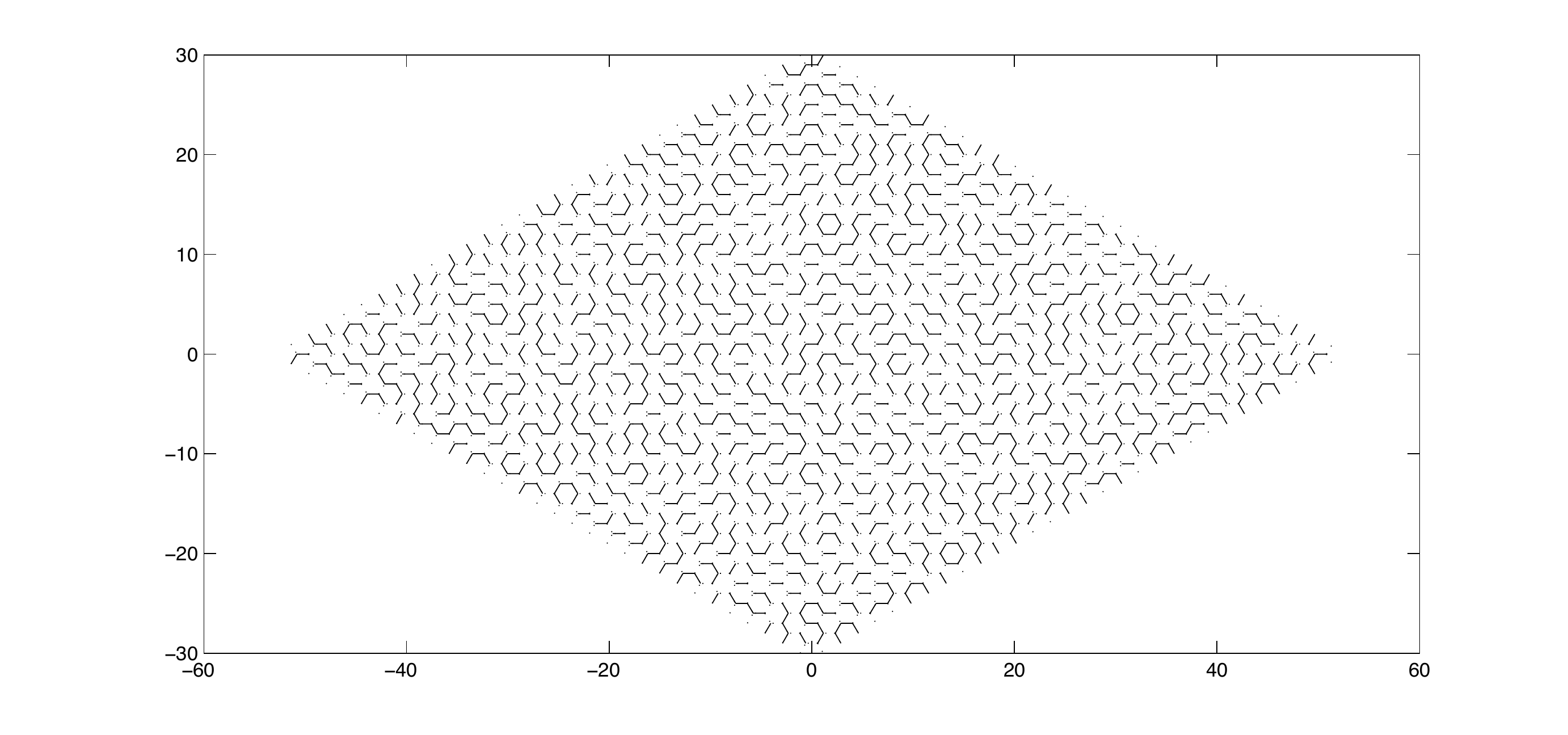}
   \caption{1-2 model configuration on the hexagonal lattice}
   \label{fig:30}
\end{figure}

The 1-2 model on the hexagonal lattice was first studied in \cite{sb}, as a constrained system whose total number of configurations can be counted in polynomial time by applying the holographic algorithm \cite{Val}. Some generalizations of the algorithm were studied in \cite{ZL3} to compute the partition function (weighted sum of all the configurations) of a weighted 1-2 model. The holographic algorithm is efficient in computing the partition function of a large class of vertex models; however, since the holographic algorithm, if applicable, gives a many-to-many correspondence between configurations of a certain vertex model and those of a dimer model on a decorated graph, it is not immediately clear how this algorithm can be applied to study the probability measures and the phase transitions of the vertex model.  A measure-preserving correspondence between 1-2 model configurations on the hexagonal lattice and perfect matchings on a decorated graph was constructed in \cite{ZL}, and a phase transition result was established in \cite{GL1}; see also \cite{GL2} for a summary. 

When assigning each local configuration (configuration restricted at each vertex) a nonnegative weight, we may define a probability measure for 1-2 model configurations on the hexagonal lattice $\HH$ such that the probability of a configuration is proportional to the product of weights of local configurations at vertices. The model is closely related to the dimer model and Ising model, see \cite{GL1}. 
In this paper, we will introduce  Markov chains to sample the 1-2 model, and study mixing time of these chains, by investigating the relation of the 1-2 model and the dimer model and by utilizing the well-known relation between the fast spacing mixing and fast temporal mixing. Note that the Ising model corresponding to the 1-2 model may not be ferromagnetic - therefore the natural stochastic domination results associated to the ferromagnetic Ising model may not be applied here. Instead of considering the single site dynamics, we study the block dynamics. When the local weights of the 1-2 model implies a lower bound on the spectral gap of the block dynamics.

The organization of the paper is as follows. In Section \ref{sc:mc}, we define the Markov chain and state the main theorems. In Section \ref{sc:ir}, we prove a few properties of the Markov chain defined in Section \ref{sc:mc}, including irreducibility, reversibility and aperiodicity. In Section \ref{sc:ct}, we review the path method and the comparison theorem which relates the mixing time of block dynamics to that of single-site dynamics; see also \cite{LPW}. In Section \ref{sc:cl}, we discuss the mixing time for the 1-2 model Markov chain on a rectangular box of the hexagonal lattice with fixed width and very large length. In Section \ref{sc:wm}, we give a sufficient condition of the strong mixing of 1-2 model configurations on an $n\times n$ box of the hexagonal lattice, which implies the fast mixing of the block dynamics (see \cite{DSVW}), and hence the fast mixing of single site dynamics by the comparison theorem.

\section{Markov Chains}\label{sc:mc}

In this section, we review the basic definitions in the theory of Markov chains including the total variation distance and the mixing time. We then define Markov chains whose state space consists of all the 1-2 model configurations on a finite subgraph of the hexagonal lattice, such that the stationary distributions of these chains are the probability measures on the 1-2 model configurations in which the probability of each configuration is proportional to the product of weights at vertices. We then state the main theorems proved in this paper.

\subsection{Total variation distance and mixing time}

\begin{definition}Let $\Omega$ be a finite state space. The \textbf{total variation distance} between two probability distributions $\mu$ and $\nu$ on $\Omega$ is defined by
\begin{eqnarray*}
\|\mu-\nu\|_{TV}=\max_{A\subset\Omega}|\mu(A)-\nu(A)|.
\end{eqnarray*}
The total variation distance also has the following expression (see Proposition 4.2 of \cite{LPW})
\begin{eqnarray*}
\|\mu-\nu\|_{TV}=\frac{1}{2}\sum_{x\in \Omega}|\mu(x)-\nu(x)|.
\end{eqnarray*}
\end{definition}

\begin{definition}Let $\Omega$ be a finite state space. Let $P$ be the transition matrix of a Markov chain whose state space is $\Omega$ and stationary distribution is $\pi$. For $t=1,2,\ldots$, let
\begin{eqnarray*}
d(t):=\max_{x\in \Omega}\|P^t(x,\cdot)-\pi\|_{TV}
\end{eqnarray*}
Then the \textbf{mixing time} is defined by
\begin{eqnarray*}
t_{mix}(\epsilon):=\min\{t:d(t)\leq \epsilon\},
\end{eqnarray*}
and
\begin{eqnarray}
t_{mix}:=t_{mix}\left(\frac{1}{4}\right).\label{mtm}
\end{eqnarray}
\end{definition}

\subsection{1-2 model Markov chains on the hexagonal lattice}\label{mch}

Observe that the hexagonal lattice $\HH$ is a bipartite graph in the sense that each vertex can be colored black and white such that vertices of the same color are not adjacent.

Let $\Lambda=(V_{\Lambda},E_{\Lambda})$ be a finite, connected subgraph of $\HH$. A boundary condition $b_{\Lambda}$ for 1-2 model configurations on $\Lambda$ is given by specifying the configurations on all the vertices outside $\Lambda$. We say a boundary condition $b_{\Lambda}$ is \df{admissible} if there exists a 1-2 model configuration on $\HH$ whose restriction on vertices outside $\Lambda$ is given by $b_{\Lambda}$. 

Assume that the boundary condition $b_{\Lambda}$ is admissible. Let $\Omega_{\Lambda,b_{\Lambda}}$ be the set of all 1-2 model configurations on $\Lambda$ with boundary condition $b_{\Lambda}$. Let $a,b,c > 0$. 
We associate the parameters $a,b,c$ to the local configurations at each vertex  
as indicated in Figure \ref{fig:sign}.

\begin{figure}[htbp]
\centerline{\includegraphics*[width=0.98\hsize]{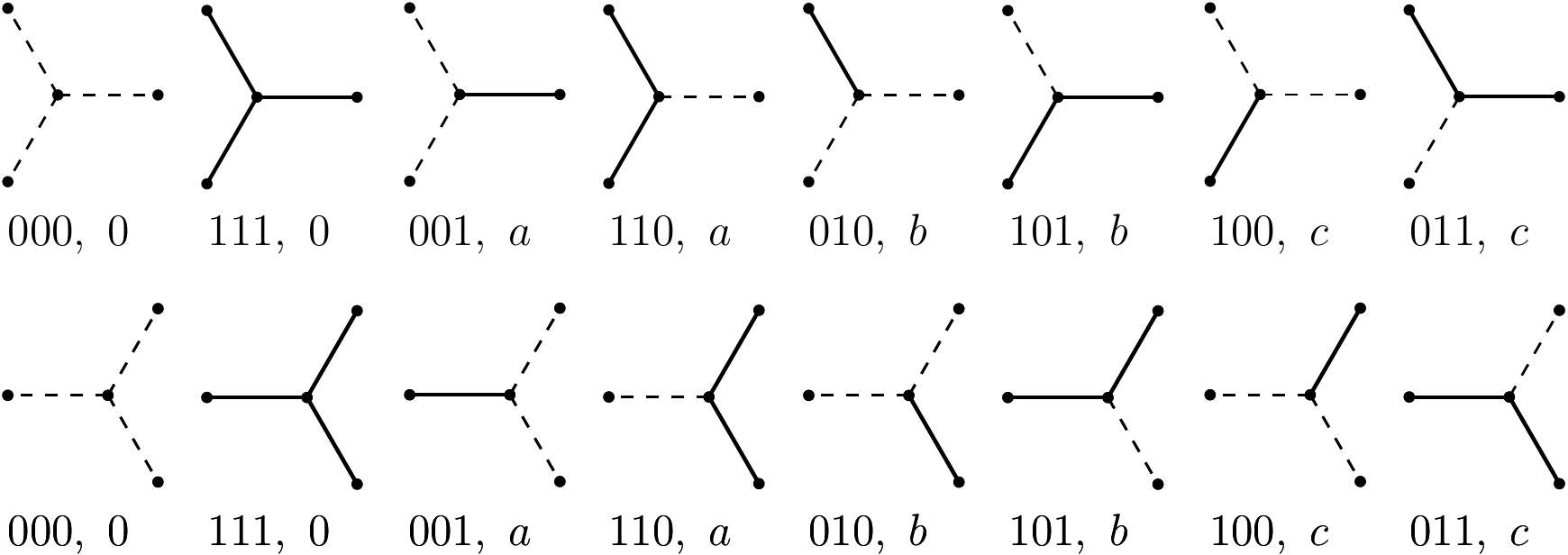}}
   \caption{The eight possible local configurations $\omega|_v$ at a vertex $v$  in the two
   cases of \emph{black} and \emph{white} vertices. The signature
   of each is given, and also the local weight $w(\omega|_v)$ associated with each instance.}
   \label{fig:sign}
\end{figure}

To each configuration $\omega\in \Omega_{\Lambda,b_{\Lambda}}$, we assign the weight 
\begin{equation}\label{eq:pf-2}
w(\omega) = \prod_{v\in V_{\Lambda}} w(\omega|_v),
\end{equation}
where $\omega|_{v}$ is one of the eight possible local configurations in Figure \ref{fig:sign} obtained by restricting $\omega$ to the vertex $v$, and $w(\omega|_{v})\in\{0,a,b,c\}$ is the weight of the local configuration. Note that if $\omega|_{v}$ does not satisfy the constraint that the vertex $v$ has one or two incident present edges, we have $w(\omega|_{v})=0$. These weights give rise to the partition
function
\begin{equation}\label{eq:pf-1}
Z=\sum_{\omega\in\Omega_{\Lambda,b_{\Lambda}}} w(\omega),
\end{equation}
which leads in turn to the probability measure  
\begin{equation}\label{eq:pm}
\mu(\omega) = \frac1Z  w(\omega), \qq\omega\in\Omega_{\Lambda,b_{\Lambda}}.
\end{equation}

 We define a Markov chain on $\Omega_{\Lambda,b_{\Lambda}}$ with transition matrix $P$ as follows. For any two 1-2 model configurations $\omega_1,\omega_2\in\Omega_{\Lambda,b_{\Lambda}}$, $P(\omega_1,\omega_2)>0$ if and only if one of the following conditions is true
\begin{enumerate}
\item $\omega_1=\omega_2$;
\item $\omega_1$ and $\omega_2$ differ at exactly one edge; see Figure \ref{fig:p2};
\item $\omega_1$ and $\omega_2$ differ at exactly one face $f$; on the boundary of $f$ both $\omega_1$ and $\omega_2$ have three present edges, such that there is a way to cyclically order the boundary edges of $f$ by $e_1,\ldots,e_6$, and $e_1,e_3,e_5$ are present (resp.\ absent) edges of $\omega_1$ (resp.\ $\omega_2$), while $e_2,e_4,e_6$ are present (resp.\ absent) edges of $\omega_2$ (resp.\ $\omega_1$). Let $g_1,\ldots g_6$ be the six incident edges of the face $f$ in cyclic order; then in both $\omega_1$ and $\omega_2$, $g_1$, $g_3$ and $g_5$ are present while $g_2$, $g_4$, $g_6$ are absent. See Figure \ref{fig:p3}.
\end{enumerate}

\begin{figure}[htbp]
\centerline{\includegraphics*{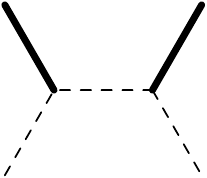}\qquad\qquad\includegraphics*{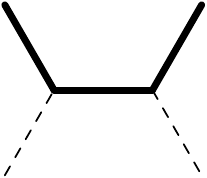}}
   \caption{Two 1-2 model configurations differ at one edge.}
   \label{fig:p2}
\end{figure}

\begin{figure}[htbp]
\centerline{\includegraphics*{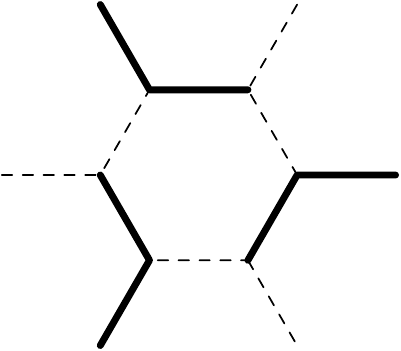}\qquad\qquad\includegraphics*{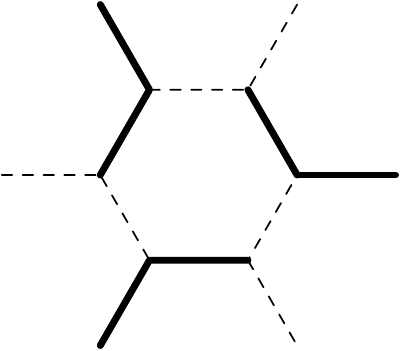}}
   \caption{Two 1-2 model configurations differ at one face.}
   \label{fig:p3}
\end{figure}

The transition matrix $P$ is a function $\Omega_{\Lambda, b_{\Lambda}}\times \Omega_{\Lambda,b_{\Lambda}}\rightarrow [0,1]$. To describe the transition matrix, we further divide the Condition (2) above into sub-cases. Assume that $\omega_1$ and $\omega_2$ differ at exactly one edge $e=(u,v)$. Let $e_1$ and $e_2$ be the two present edges in both $\omega_1$ and $\omega_2$, where $e_1$ (resp.\ $e_2$) shares the endpoint $u$ (resp.\ $v$) with $e$. The following cases might occur
\begin{enumerate}[label=(\alph*)]
\item $e_1$ and $e_2$ are not parallel;
\item $e$ is a horizontal edge, both $e_1$ and $e_2$ are NW-SE edges; $e$ is present in $\omega_1$ but absent in $\omega_2$; 
\item $e$ is a horizontal edge, both $e_1$ and $e_2$ are NW-SE edges; $e$ is present in $\omega_2$ but absent in $\omega_1$; 
\item $e$ is a horizontal edge, both $e_1$ and $e_2$ are NE-SW edges; $e$ is present in $\omega_1$ but absent in $\omega_2$; 
\item $e$ is a horizontal edge, both $e_1$ and $e_2$ are NE-SW edges; $e$ is present in $\omega_2$ but absent in $\omega_1$; 
\item $e$ is a NW-SE edge, both $e_1$ and $e_2$ are horizontal edges; $e$ is present in $\omega_1$ but absent in $\omega_2$; 
\item $e$ is a NW-SE edge, both $e_1$ and $e_2$ are horizontal edges; $e$ is present in $\omega_2$ but absent in $\omega_1$; 
\item $e$ is a NW-SE edge, both $e_1$ and $e_2$ are NE-SW edges; $e$ is present in $\omega_1$ but absent in $\omega_2$; 
\item $e$ is a NW-SE edge, both $e_1$ and $e_2$ are NE-SW edges; $e$ is present in $\omega_2$ but absent in $\omega_1$; 
\item $e$ is a NE-SW edge, both $e_1$ and $e_2$ are NW-SE edges; $e$ is present in $\omega_1$ but absent in $\omega_2$; 
\item $e$ is a NE-SW edge, both $e_1$ and $e_2$ are NW-SE edges; $e$ is present in $\omega_2$ but absent in $\omega_1$; 
\item $e$ is a NE-SW edge, both $e_1$ and $e_2$ are horizontal edges; $e$ is present in $\omega_1$ but absent in $\omega_2$; 
\item $e$ is a NE-SW edge, both $e_1$ and $e_2$ are horizontal edges; $e$ is present in $\omega_2$ but absent in $\omega_1$; 
\end{enumerate}
Let $E_{\Lambda}$, $F_{\Lambda}$ be the edge set, face set of of $\Lambda$, respectively. Let 
\begin{eqnarray*}
C_{\Lambda}:=\frac{\min\{a,b,c\}}{2(|E_{\Lambda}|+|F_{\Lambda}|)\max\{a,b,c\}}
\end{eqnarray*}
\begin{eqnarray}
&&P(\omega_1,\omega_2)\label{tmt}\\
&=&\left\{\begin{array}{cc} C_{\Lambda}& \mathrm{if}\ \omega_1, \omega_2\ \mathrm{satisfy\ Condition\ (3)\ or (2)(a)} ;\\
\frac{C_{\Lambda}b}{c}& \mathrm{if}\ \omega_1, \omega_2\ \mathrm{satisfy\ Condition\ (2)(b)\ or\ (2)(e)}; \\
\frac{C_{\Lambda}c}{b}& \mathrm{if}\ \omega_1, \omega_2\ \mathrm{satisfy\ Condition\ (2)(c)\ or\ (2)(d)}; \\
\frac{C_{\Lambda}a}{c}& \mathrm{if}\ \omega_1, \omega_2\ \mathrm{satisfy\ Condition\ (2)(f)\ or\ (2)(i)}; \\
\frac{C_{\Lambda}c}{a}& \mathrm{if}\ \omega_1, \omega_2\ \mathrm{satisfy\ Condition\ (2)(g)\ or\ (2)(h)}; \\
\frac{C_{\Lambda} b}{a}& \mathrm{if}\ \omega_1, \omega_2\ \mathrm{satisfy\ Condition\ (2)(j)\ or\ (2)(m)}; \\
\frac{C_{\Lambda}a}{b}& \mathrm{if}\ \omega_1, \omega_2\ \mathrm{satisfy\ Condition\ (2)(k)\ or\ (2)(l)}; \\
0& \mathrm{if}\ \omega_1\neq \omega_2,\ \mathrm{and}\ \omega_1,\omega_2\ \mathrm{satisfy\ neither\ (2)\ nor\ (3)};\\
 1-\sum_{\omega'\in \Omega_{\Lambda,b_{\Lambda}},\omega'\neq \omega}P(\omega,\omega'),&\mathrm{if}\ \omega_1=\omega_2=\omega. \end{array}\right.\notag
\end{eqnarray}
Note that the value of $C_{\Lambda}$ guarantees that 
\begin{eqnarray*}
1-\sum_{\omega'\in \Omega_{\Lambda,b_{\Lambda}},\omega'\neq \omega}P(\omega,\omega')\geq \frac{1}{2}.
\end{eqnarray*}

Here are the main theorems proved in this paper concerning mixing times of Markov chains of for the 1-2 model on the hexagonal lattice.

\begin{theorem}\label{m1}Let $\Lambda_{k,n}$ be a rectangular $k\times n$ box of the hexagonal lattice, where $k$ is the fixed width of the rectangle, and $n(\gg k)$ is the length of the rectangle.  Let $b_{k,n}$ be an admissible boundary condition for 1-2 model configurations on $\Lambda_{k,n}$ given by specifying the states of all edges outside $\Lambda_{k,n}$. Let $\Omega_{\Lambda_{k,n},b_{k,n}}$ be the set of all the 1-2 model configurations on $\Lambda_{k,n}$ with boundary condition $b_{k,n}$. The Markov chain with state space $\Omega_{\Lambda_{k,n},b_{k,n}}$ and transition matrix $P$ defined by (\ref{tmt}) has mixing time $t_{mix}$ satisfying
\begin{eqnarray*}
B(k)n\leq t_{mix}\leq C(k)n^2,
\end{eqnarray*}
where $C(k)>0, B(k)>0$ are a constant depending on $k$ and independent of $n$, and $t_{mix}$ is defined by (\ref{mtm}).
\end{theorem}

Theorem \ref{m1} is proved in Section \ref{sc:cl}.

\begin{theorem}\label{m2}Let $\Lambda_{n}$ be an $n\times n$ box of the hexagonal lattice $\HH$.  Let $b_{\Lambda_n}$ be an admissible boundary condition for 1-2 model configurations on $\Lambda_{n}$ given by specifying the states of all edges outside $\Lambda_{n}$. Let $\Omega_{\Lambda_n,b_{\Lambda_n}}$ be the set of all the 1-2 model configurations on $\Lambda_{n}$ with boundary condition $b_{\Lambda_n}$. Assume that the local weights $a,b,c>0$ are such that condition $F\left(m,\frac{1}{15}\right)$ holds (see Definition \ref{d62}) for some fixed positive integer $m$, then when $n$ is sufficiently large the Markov chain with state space $\Omega_{\Lambda_n,b_{\Lambda_n}}$ and transition matrix $P$ defined by (\ref{tmt}) has mixing time $t_{mix}$ satisfying
\begin{eqnarray*}
B n^2\leq t_{mix}\leq C n^4,
\end{eqnarray*}
where $B,C>0$ is a constant independent of $n$, and $t_{mix}$ is defined by (\ref{mtm}).
\end{theorem}

Theorem \ref{m2} is proved in Section \ref{sc:wm}.

\section{Reversibility and Irreducibility}\label{sc:ir}

In this section, we prove properties of the Markov chain defined in Section \ref{mch}. These properties are important for later analysis of the chain, including the application of the comparison method.  In Proposition \ref{ir}, we discuss the irreducibility of the chain. In Proposition \ref{re}, we discuss the reversibility and the stationary distribution of the chain. In Proposition \ref{ap}, we discuss the aperiodicity and eigenvalues of the transition matrix.

\begin{proposition}\label{ir}Assume that $\Lambda=\Lambda_{m,n}$, where $\Lambda_{m,n}$ is an $m\times n$ box of the hexagonal lattice. Then the Markov chain defined by (\ref{tmt}) is irreducible on $\Omega_{\Lambda,b_{\Lambda}}$, for any admissible boundary condition $b_{\Lambda}$ and parameters $a,b,c>0$.
\end{proposition}

Before proving Proposition \ref{ir}, we first recall that each 1-2 model configuration $\omega$ on $\HH$ corresponds to a perfect matching $D_{\omega}$ on a decorated graph $\HH_{\Delta}=(V_{\Delta},E_{\Delta})$, constructed in \cite{ZL}. See Figure \ref{fig:12con}, where the hexagonal lattice $\HH$ is represented by solid lines of the left graph, and dashed lines on the right graph; the 1-2 model configuration $\omega$ is represented by thick solid lines on the left graph; the decorated graph $\HH_{\Delta}$ is represented by solid lines on the right graph; and the dimer configuration $D_{\omega}$ is represented by thick solid lines on the right graph. 

We now explain the correspondence from $\omega$ to $D_{\omega}$. 
The vertex set $V$ of $\HH$ is a proper subset of the vertex set $V_{\Delta}$ of $\HH_{\Delta}$. Each vertex $v\in V$ is incident to three edges in $E_{\Delta}$. More precisely, the three incident edges of $v\in V$ in $E_{\Delta}$ are bisectors of the three angles of $\HH$ at $v$. A bisector edge in $E_{\Delta}$ (i.e. an edge in $E_{\Delta}$ incident to a vertex $v\in V$) is present in $D_{\omega}$ if and only if the two sides of the angle in $\HH$ have the same state in $\omega$; i.e. either both are present or both are absent. Given the constraint that $v\in V$ has one or two incident present edges in $\omega$, the vertex $v$ has exactly one incident present edge in $D_{\omega}$, according to the correspondence described above. The configuration on bisector edges can be uniquely extended to a dimer configuration on $\HH_{\Delta}$, since each face has an even number of present bisector edges in the configuration.

\begin{figure}[htbp]
\includegraphics[width=0.9\textwidth]{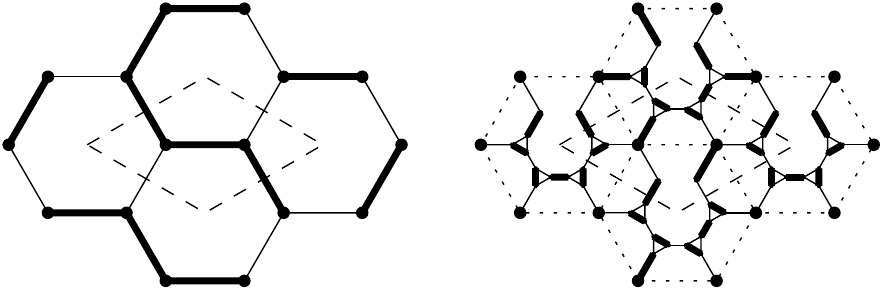}
\caption{Part of a \ot configuration $\omega$ on $\HH$ is illustrated on the left graph, and the corresponding dimer (sub)configuration $D_{\omega}$
on $\HH_{\De}$ is illustrated on the right graph.
When two edges with a common vertex of $\HH$ have the same state in $\omega$,
the corresponding `bisector edge' is present in $D_{\omega}$.
The states of the bisector edges determine the dimer configuration on the rest of $\HH_{\De}$.}
\label{fig:12con}
\end{figure}

The correspondence described above is 2-to-1 because if $\omega'$ is the 1-2 model configuration obtained from $\omega$ by making all the present edges of $\omega$ absent, and all the absent edges of $\omega$ present, then $\omega$ and $\omega'$ correspond to the same dimer configuration on $\HH_{\Delta}$. However, restricted to 1-2 model configurations on $\HH_{\Lambda}$ with given admissible boundary condition $b_{\Lambda}$ the above correspondence is a bijection.

A \textbf{cycle} in $\HH_{\Delta}$ is an alternating sequence $v_0,e_1,v_2,\ldots,v_{n-1},e_n,v_n$ of vertices and edges of $\HH_{\Delta}$, such that
\begin{itemize}
\item $v_i\in V_{\Delta}$, for $0\leq i\leq n$; and 
\item $e_i=(v_{i-1},v_i)\in E_{\Delta}$, for $1\leq i\leq n$; and
\item $v_0=v_n$; and
\item $v_i\neq v_j$, if $(i,j)\notin\{ (0,n),(n,0)\}$.
\end{itemize}
We consider two special types of cycles of $\HH_{\Delta}$. A \df{Type-I cycle} of $\HH_{\Delta}$ is a cycle enclosing exactly one edge of the hexagonal lattice, while a \df{Type-II cycle} of $\HH_{\Delta}$ is a cycle enclosing exactly one face of $\HH$.  See Figure \ref{fig:t12c} for examples of Type-I and Type-II cycles. Note that for each edge of $\HH$, there are different Type-I cycles enclosing the edge. Two Type-I cycles of $\HH_{\Delta}$ enclosing the same edge of $\HH$ may differ by a few triangles, but they occupy exactly the same set of edges incident to vertices of $\HH$. Indeed, if we contract every triangle face of $\HH_{\Delta}$ into a degree-3 vertex, any two Type-I cycles enclosing the same edge have the same image under such a contraction. Similar results hold for different Type-II cycles of $\HH_{\Delta}$ enclosing the same face of $\HH$.

\begin{figure}[htbp]
\includegraphics[width=0.6\textwidth]{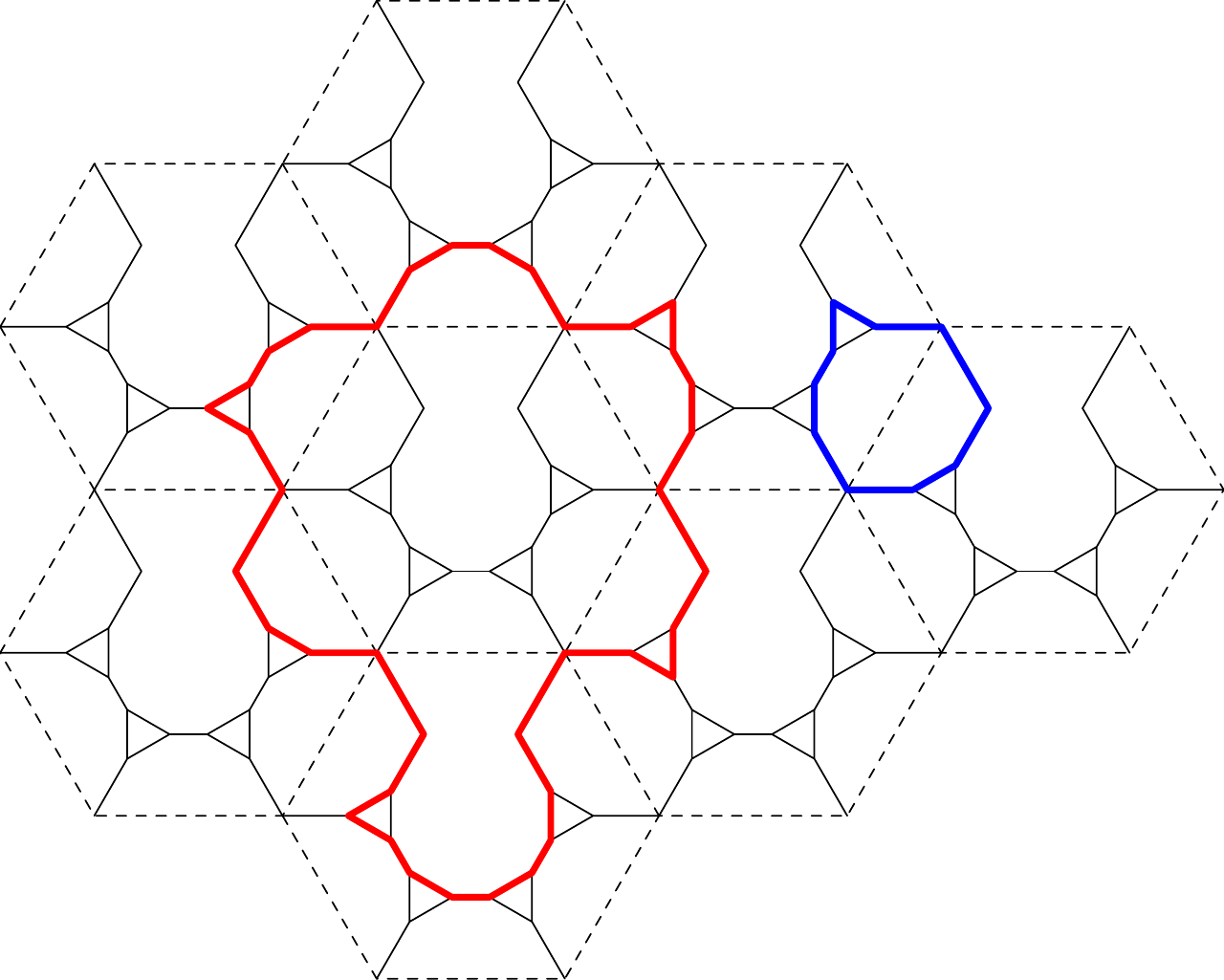}
\caption{Type-I cycle and Type-II cycle: Dashed lines represent the hexagonal lattice, solid black lines represent the decorated graph $\HH_{\Delta}$, red lines represent a Type-II cycle, and blue lines represent a Type-I cycle.}
\label{fig:t12c}
\end{figure}

Let $\tilde{b}_{\Lambda}$ be the boundary condition of dimer configurations obtained from the boundary condition of the 1-2 model configuration $b_{\Lambda}$, by the correspondence between 1-2 model configurations on $\HH$ and dimer configurations on $\HH_{\Delta}$ described above.

\begin{lemma}\label{sm} Let $\Lambda$ be a finite subgraph of $\HH$. Let $\omega_1,\omega_2\in \Omega_{\Lambda,b_{\Lambda}}$ be two 1-2 model configurations on $\Lambda$ with boundary condition $b_{\Lambda}$. Let $D_1$ (resp. $D_2$) be the dimer configuration corresponding to $\omega_1$ (resp. $\omega_2$). Consider the Cases (1)-(3) of $\omega_1$ and $\omega_2$ as given in Section \ref{mch}.
\begin{enumerate}[label=\Roman*.]
\item The 1-2 model configurations $\omega_1$ and $\omega_2$ differ at exactly one edge $e\in E$ as in Case (2), if and only if $D_1\cup D_2$ consists of doubled edges and a unique Type-I cycle enclosing the edge $e$.
\item The 1-2 model configurations $\omega_1$ and $\omega_2$ differ at exactly one face $f$ of $\HH$ as in Case (3), if and only if $D_1\cup D_2$ consists of doubled edges and a unique Type-II cycle enclosing the face $f$.
\end{enumerate}

\end{lemma}

\begin{proof}Let $\HH_{\Delta,\Lambda}$ be the subgraph of $\HH_{\Delta}$ corresponding to $\Lambda$. Since both $\omega_1$ and $\omega_2$ have the same boundary condition $b_{\Lambda}$, both $D_1$ and $D_2$ have the same boundary condition $\tilde{b}_{\Lambda}$. The double dimer configuration $D_1\cup D_2$ in $\HH_{\Delta,\Lambda}$ consists of doubled edges and cycles, where the cycles do not intersect the boundary.

First let us consider the case when $\omega_1$ and $\omega_2$ differ at exactly one edge $e=(u,v)$ as in Case (2), where $u,v\in V$ are two vertices of $\HH$. 
For any vertex $w\in V\setminus\{u,v\}$, each one of the three incident edges of $w$ in $\HH$ has the same state in $\omega_1$ and $\omega_2$. Therefore each one of the three incident edges of $w$ in $\HH_{\Delta}$ has the same state in $D_1$ and $D_2$, by the correspondence between $(\omega_1,\omega_2)$ and $(D_1,D_2)$ described before. 
 Then in $D_1\cup D_2$, any vertex $w\in V\setminus \{u,v\}$ is incident to a doubled edge. 
 
 However, at the vertex $u$ or $v$, its incident edges in $\HH_{\Delta}$ have different states in $D_1$ and $D_2$ since its incident edges in $\HH$ have different states in $\omega_1$ and $\omega_2$. Therefore both $u$ and $v$ are along a cycle in $D_1\cup D_2$. See Figure \ref{fig:p2d}. There must be a unique Type-I cycle passing through both $u$ and $v$ since no other vertices of $\HH$ are along a cycle in $D_1\cup D_2$.
 
 \begin{figure}[htbp]
\centerline{\includegraphics*{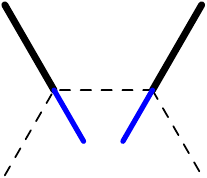}\qquad\qquad\includegraphics*{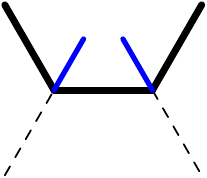}\qquad\qquad\includegraphics*{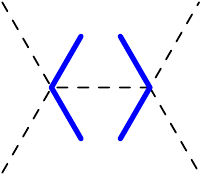}}
   \caption{Two 1-2 model configurations $\omega_1$ and $\omega_2$ differ at one edge. In the left (\resp. middle) graph, black solid lines represent $\omega_1$ (resp. $\omega_2$); blue solid lines represent $D_1$ (resp. $D_2$). In the right graph, blue solid lines represent $D_1\cup D_2$.}
   \label{fig:p2d}
\end{figure}
 
 Conversely, if $D_1\cup D_2$ consists of doubled edges and a unique Type-I cycle enclosing an edge $e=(u,v)$, then at any vertex $w\in V\setminus\{u,v\}$, the incident edges of $w$ in $\HH$ have the same state in $\omega_1$ and $\omega_2$ since the incident edges of $w$ in $\HH_{\Delta}$ have the same state in $D_1$ and $D_2$ and moreover, both $\omega_1$ and $\omega_2$ have the same boundary condition $b_{\Lambda}$.  Therefore each edge in $E\setminus\{e\}$ has the same state in $\omega_1$ and $\omega_2$, since it is incident to at least one vertex in $V\setminus \{u,v\}$. However, $\omega_1$ and $\omega_2$ cannot have the same state at the edge $e$, since if so, $D_1$ and $D_2$ would be identical. This completes the proof of Part I. of the lemma.
 
 Now we prove Part II. Assume that $f$ is a face of $\HH$ consisting of vertices $v_1,\ldots,v_6$ and edges $e_1=(v_6,v_1), e_1=(v_1,v_2),\ldots, e_6=(v_5,v_6)$. 
 
 Assume that $\omega_1$ and $\omega_2$ differ at exactly the face $f$ as described in (3). By the same arguments as in the proof of I., $D_1\cup D_2$ has doubled edges incident to every vertex in $w\in V\setminus \{v_1, v_2,\ldots,v_6\}$, and each one of $v_1,\ldots,v_6$ is along a cycle in $D_1\cup D_2$. Moreover, any cycle in $D_1\cup D_2$ cannot intersect the interior of $F$. That is because for $1\leq i\leq 5$, $e_i$ and $e_{i+1}$ have different states in $\omega_1$ (resp.\ $\omega_2$); hence the bisector edge of the angle with sides $e_i$ and $e_{i+1}$ is not present in $D_1$ (resp.\ $D_2$). Similarly, the bisector edge of the angle with sides $e_1$ and $e_6$ is present in neither $D_1$ nor $D_2$. 
 See Figure \ref{fig:pd3}.
 \begin{figure}[htbp]
\centerline{\includegraphics*{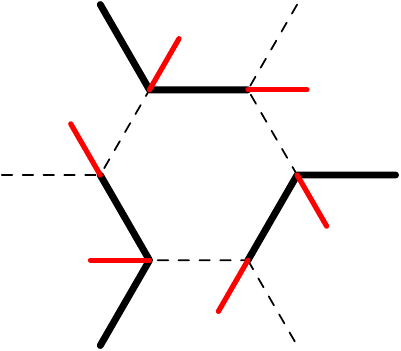}\qquad\qquad\includegraphics*{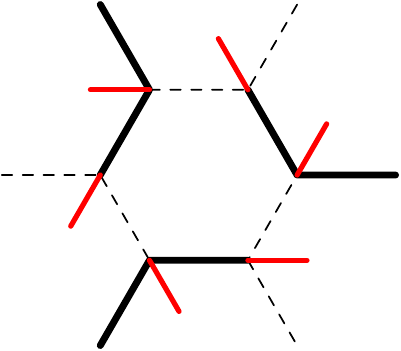}\qquad\qquad\includegraphics*{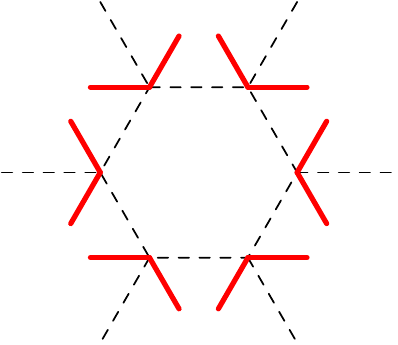}}
   \caption{Two 1-2 model configurations $\omega_1$, $\omega_2$ differ at one face. In the left (\resp. middle) graph, black solid lines represent $\omega_1$ (resp. $\omega_2$); red solid lines represent $D_1$ (resp. $D_2$). In the right graph, red solid lines represent $D_1\cup D_2$.}
   \label{fig:pd3}
\end{figure}
Therefore $D_1\cup D_2$ consists of doubled edges and a unique Type-II cycle enclosing the face $f$.

Conversely, assume that $D_1\cup D_2$ consists of doubled edges and a unique Type-II cycle enclosing the face $f$. By the same arguments as in the proof of Part I., all the edges in $E\setminus \{e_1,e_2,\ldots,e_6\}$ have the same state in $\omega_1$ and $\omega_2$. Since the unique Type-II cycle in $D_1\cup D_2$ encloses the face $F$, the cycle does not intersect the interior of $f$. This implies for $1\leq i\leq 5$, $e_i$ and $e_{i+1}$ have different states in $\omega_1$ (resp.\ $\omega_2$). Also $e_1$ and $e_6$ have different states in $\omega_1$ (resp.\ $\omega_2$). Moreover, at each one of the vertices $v_1,v_2,\ldots,v_6$, the two incident present edges in $D_1\cup D_2$ are the two incident edges of the vertex not in $f$. Since each edge of $\HH$ outside $f$ has the same state in $\omega_1$ and $\omega_2$, we deduce that $\omega_1$ has $e_1,e_3,e_5$ present, $e_2,e_4,e_6$ absent; while $\omega_2$ has $e_2,e_4,e_6$ present, $e_1,e_3,e_5$ absent, or vice versa. This completes the proof of II.
\end{proof}

Let $C$ be an even-length cycle of $\HH_{\Delta}$. We label the edges of $C$ cyclically by $e_1,e_2,\ldots,e_{2n}$. Let $D$ be a dimer configuration on $\HH_{\Delta}$, such that $e_1,e_3,\ldots,e_{2n-1}$ are present in $D$, while $e_2,e_4,\ldots,e_{2n}$ are absent. By \textbf{rotating configurations of $D$ along $C$ to alternate edges}, we obtain a new dimer configuration $D'$ on $\HH_{\Delta}$, such that $D'=D$ for any edge $e\in E_{\Delta}\setminus\{e_1,\ldots,e_{2n}\}$; while in $D'$, $e_2,e_4,\ldots,e_{2n}$ are present and $e_1,e_3,\ldots, e_{2n-1}$ are absent. See Figure \ref{fig:rot} for an example of rotating dimer configurations along a Type-I cycle to alternate edges.
 \begin{figure}[htbp]
\centerline{\includegraphics*{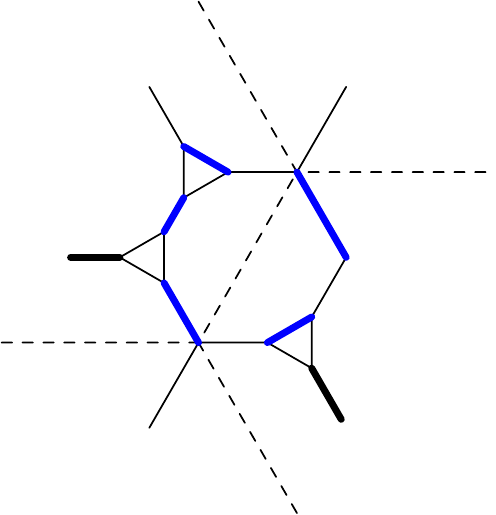}\qquad\qquad\includegraphics*{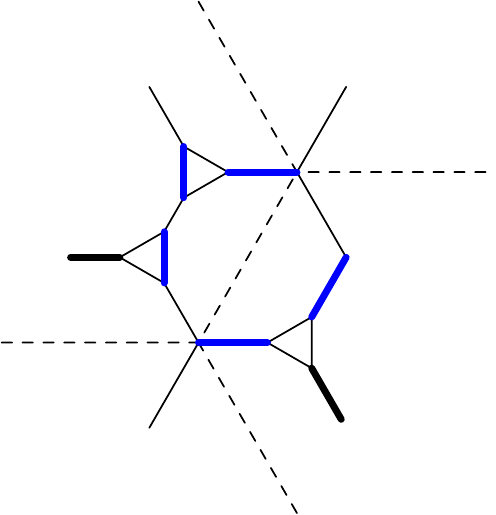}}
   \caption{Two dimer configurations $D$ and $D'$. In the left graph, thick lines represent $D$. In the right graph, thick lines represent $D'$. $D'$ is obtained from $D$ by rotating configurations along a Type-I cycle to alternate edges. Thick black lines represent configurations that are the same in both $D$ and $D'$, and thick blue lines represent configurations that are different in $D$ and $D'$.}
   \label{fig:rot}
\end{figure}
By Lemma \ref{sm}, rotating the dimer configuration along a Type-I (resp.\ Type-II) cycle corresponds to changing 1-2 model configurations from $\omega_1$ to $\omega_2$ such that $\omega_1$ and $\omega_2$ satisfy Condition (2) (resp.\ Condition (3)).

\bigskip
\noindent\textbf{Proof of Proposition \ref{ir}.} Let $\omega_1,\omega_2\in \Omega_{\Lambda,b_{\Lambda}}$ be two arbitrary 1-2 model configurations on $\HH_{\Lambda}$ with boundary condition $b_{\Lambda}$. To show that the Markov chain with transition matrix $P$ defined by (\ref{tmt}) is irreducible, it suffices to show that $\omega_2$ can be obtained from $\omega_1$ by finitely many manipulations of configurations on single edges or single faces as in (2) or (3). By Lemma \ref{sm}, it suffices to show that for any two dimer configurations $D_1$, $D_2$ of $\HH_{\Delta,\Lambda}$ with boundary condition $\tilde{b}_{\Lambda}$, $D_2$ can be obtained from $D_1$ by moving present edges to alternate edges along finitely many Type-I and Type-II cycles.

Note that $D_1\cup D_2$ is a disjoint union of doubled edges and even-length cycles, such that the cycles are all in $\HH_{\Delta,\Lambda}$, and never cross the boundary of $\HH_{\Delta, \Lambda}$. The symmetric difference $D_1\Delta D_2$ is a disjoint union of even length cycles.

We say a cycle $C$ in $D_1\Delta D_2$ is \df{contractible} if $C$ does not enclose any other cycle in $D_1\Delta D_2$. More precisely, if we identify the cycle $C$ with its embedding into the plane $\RR^2$, the complement of $C$ in $\RR^2$ - denoted by $\RR^2\setminus C$, has exactly one bounded component and one unbounded component. The cycle $C$ is contractible if and only if the bounded component of $\RR^2\setminus C$ does not include any cycle in $D_1\Delta D_2$.
We will eliminate contractible cycles in $D_1\Delta D_2$ one by one by changing configurations of $D_1$ or $D_2$ to alternate edges on finitely many Type-I or Type-II cycles.  

Let $\mathcal{C}(D_1\Delta D_2)$ be the set of all cycles in $D_1\Delta D_2$. For each cycle $C\in\mathcal{C}(D_1\Delta D_2)$, let $\mathcal{N}(C)$ be the number of edges of $\HH$ enclosed by $C$. More precisely, if we identify $C$ and edges of $\HH$ with its embedding into $\RR^2$, where edges of $\HH$ are closed line segments, $\mathcal{N}(C)$ is the number of edges of $\HH$ whose interior lie in the bounded component of $\RR^2\setminus C$. For example, if $C$ is a Type-I cycle (resp.\ Type-II) cycle, then $\mathcal{N}(C)=1$ (resp. $\mathcal{N}(C)=6$). For any two dimer configurations $D_1,D_2$ in $\HH_{\Delta,\Lambda}$ with boundary condition $\tilde{b}_{\Lambda}$, define the \textbf{distance} $d(D_1,D_2)$ by
\begin{eqnarray}
d(D_1,D_2)=\sum_{C\in\mathcal{C}(D_1\Delta D_2)}\mathcal{N}(C).\label{d12}
\end{eqnarray}
Obviously $D_1=D_2$ if and only if $d(D_1,D_2)=0$.

Let $C$ be a contractible cycle of $D_1\Delta D_2$. The complement $\RR^2\setminus C$ has exactly one bounded component, denoted by $S$. Let $v$ be a vertex of $\HH$ along $C$. The following cases might occur
\begin{numlist}
\item At $v$, $S$ forms an angle of size $\frac{2\pi}{3}$.
\item At $v$, $S$ forms an angle of size $\frac{4\pi}{3}$.
\end{numlist}
In  Lemmas \ref{l33} and \ref{l34} below, we will show that in either case, there is always a way to rotate configurations of $D_1$ or $D_2$ to alternate edges along finitely many Type-I and/or Type II cycles, such that the distance of the two new dimer configurations after the configuration change strictly decreases, compared to $d(D_1,D_2)$. Since $d(D_1,D_2)$ is always finite and integer-valued, after finitely many steps of configuration changes, $D_1$ becomes $D_1^{(i)}$ and $D_2$ becomes $D_2^{(j)}$, and we will obtain $D_1^{(i)}=D_2^{(j)}$. This completes the proof. \hfill$\Box$

\begin{lemma}\label{l33} Let $C$ be a contractible cycle in $D_1\Delta D_2$.
If there exists a vertex $v$ along $C$ satisfying Condition 1., then we can change configurations of $D_1$ or $D_2$ to alternate edges on finitely many Type-I cycles, such that the distance of the two new dimer configurations after the configuration change strictly decreases, compared to $d(D_1,D_2)$.
\end{lemma}

\begin{proof}

Assume that $v$ is a vertex of $\HH$ along $C$ satisfying Condition 1. Let $e_1$ be the edge of $\HH$ incident to $v$ inside $S$. Let $w$ be the other endpoint of $e_1$. Let $C_1$ be the Type-I cycle enclosing $e_1$. The following cases might occur.
\begin{romlist}
\item In at least one of $D_1$ and $D_2$, both incident present edges of $v$ and $w$ are along $C_1$.
\item In neither $D_1$ nor $D_2$, both incident present edges of $v$ and $w$ are along $C_1$.
\end{romlist}

Note that under Condition 1., the present edges incident to $v$ in $D_1$ and $D_2$ are exactly the two incident edges of $v$ along $C_1$.

In Case (i), without loss of generality, assume that in $D_1$, both incident present edges of $v$ and $w$ are along $C_1$. Then we change the configurations of $D_1$ to alternate edges along some Type-I cycle $C_1'$ enclosing $e_1$, and obtain a configuration $D_1^{(1)}$. Here, $C_1'$ and $C_1$ are both Type-I cycles enclosing $e_1$, they may differ by a few triangles.
 Note that
\begin{eqnarray*}
d(D_1^{(1)},D_2)=d(D_1,D_2)-1,
\end{eqnarray*}
 since $e_1$ is enclosed in one more cycles in $D_1\Delta D_2$ than $D_1^{(1)}\Delta D_2$. 

In Case (ii), at $w$ both $D_1$ and $D_2$ occupy the same incident edge in $\HH_{\Delta}$. Let $e_1,e_2,e_3$ be the three incident edges of $w$ in $\HH$. Let $C_2$ and $C_3$ be two Type-I cycles enclosing $e_2$ and $e_3$, respectively. Let $w,v_2$ (resp. $w, v_3$) be the two endpoints of $e_2$ (resp. $e_3$). Again Case (i) or Case (ii) might occur with $C_1$ replaced by $C_2$. 

In Case (ii), if Case (i) occurs when $C_1$ is replaced by $C_2$, without loss of generality, assume that in $D_2$, both incident edges of $v_2$ and $w$ are along $C_2$. Then we change configurations of $D_2$ to alternate edges along some Type-I cycle $C_2'$ enclosing $e_2$ and obtain a configuration $D_2^{(1)}$. Then we change configurations of $D_2^{(1)}$ to alternate edges along some Type-I cycle $C_1''$ enclosing $e_1$, and obtain a configuration $D_2^{(2)}$. We have
\begin{eqnarray*}
d(D_1,D_2^{(2)})=d(D_1,D_2)-2,
\end{eqnarray*}
since $e_1$ and $e_2$ are enclosed in one more cycle in $D_1\Delta D_2$ than $D_1\Delta D_2^{(2)}$.

In Case (ii), if Case (ii) occurs when $C_1$ is replaced by $C_2$, then at the vertex $v_2$, both $D_1$ and $D_2$ has the same incident present edge. Indeed, this implies that $e_1$, all the incident edges of $w$ and $v_2$ are enclosed by the original cycle $C$. We continue to explore the incident edges of $v_2$ other than $e_2$ until we find and edge $e_k$ of $\HH$, and a Type-I cycle $C_k$ surrounding $e_k$, such that Case (i) occurs when $C_1$ is replaced by $C_k$. This is always possible since $C$ is a finite cycle, we will finally find edges not enclosed by $C$ by the exploration process described above. Then we change configurations of $D_1$ or $D_2$ on finitely many Type-I cycles $C_k', C_{k-1}',\ldots, C_1'$, such that after the configuration change, $D_1$ becomes $D_1^{(i)}$, $D_2$ becomes $D_2^{(j)}$ and moreover,
\begin{eqnarray*}
 d(D_1^{(i)}\Delta D_2^{(j)})<d(D_1,D_2). 
 \end{eqnarray*}
 Here for $1\leq s\leq k$, $C_s'$ is a Type-I cycle enclosing the same edge of $\HH$ as $C_s$.
 \end{proof}
 
 \begin{lemma}\label{l34}Let $C$ be a contractible cycle in $D_1\Delta D_2$.
If every vertex along $C$ satisfies Condition 2., then we can change configurations of $D_1$ or $D_2$ to alternate edges on finitely many Type-I and/or Type-II cycles, such that the distance of two new dimer configurations after the configuration change strictly decreases, compared to $d(D_1,D_2)$.
\end{lemma}

\begin{proof}

Now we assume that every vertex of $\HH$ along $C$ satisfies Condition 2. Again let $v$ be a vertex of $S$ along $C$. Let $e_1=\langle v,w \rangle$ be an edge of $\HH$ in $S$. Let $C_1$ be a Type-I cycle enclosing $e_1$. Assume that the incident present edge of $v$ in $D_1$ is along $C_1$. The following cases might occur.
\begin{Alist}
\item If the incident present edge at $w$ in $D_1$ is also along $C_1$, we can change configurations of $D_1$ along a Type-I cycle $C_1'$ enclosing $e_1$ to alternate edges, and obtain a new configuration $D_1^{(1)}$, such that
\begin{eqnarray*}
d(D_1^{(1)},D_2)=d(D_1,D_2)-1,
\end{eqnarray*}
since $e_1$ is enclosed in one more cycles in $D_1\Delta D_2$ than $D_1^{(1)}\Delta D_2$. 
\item If the incident present edge at $w$ in $D_1$ is not along $C_1$, let $p_1$ be the incident present edge of $w$ in $D_1$. Let $p_1$ be the incident present edge of $w$ in $D_1$. Let $e_2=(w,v_2)$ be an incident edge of $w$ in $\HH$ enclosed by $C$. Note that we can always find such an $e_2$ that $e_2\neq e_1$ because of the reasons listed below.
\begin{problist}
\item If $p_1$ is also present in $D_2$, then all the three incident edges of $w$ in $\HH$ are enclosed by $C$.
\item If $p_1$ is not present in $D_2$, then $p_1\in C$ and $w\in C$, since $C$ is a contractible cycle. Recall that $S$ forms an angle of size $\frac{4\pi}{3}$ at $w$, therefore two incident edges of $w$ are in S and enclosed by $C$.
\end{problist}
Let $C_2$ be a Type-I cycle enclosing $e_2$. Obviously we have $p_1\in C_2$. 
\begin{Romlist}
\item If Case A. occurs when $(w,C_1)$ is replaced by $(v_2,C_2)$, then we rotate configurations in $D_1$ to alternate edges along Type-I cycles $C_1',C_2'$ (here $C_1',C_2'$ are Type-I cycles enclosing $e_1$, $e_2$, respectively), and obtain a new configuration $D_1^{(1)}$. We have
\begin{eqnarray*}
d(D_1^{(1)},D_2)=d(D_1,D_2)-2,
\end{eqnarray*}
because the edges $e_1,e_2$ are enclosed in one more cycle in $D_1\Delta D_2$ compared to $D_1^{(1)}\Delta D_2$.
\item If Case A. does not occur when $(w,C_1)$ is replaced by $(v_2,C_2)$, we repeat the similar process. In general, we have an induction process. We make the following induction hypothesis.
\begin{itemize}
\item Assume that for $k\geq 2$, we find a sequence of distinct edges $e_1=(v_0,v_1), e_1=(v_1,v_2),\ldots, e_k=(v_{k-1},v_k)$ enclosed by C. Let $C_i$, $1\leq i\leq k$, be a Type-I cycle enclosing $e_i$. Assume that for $1\leq i\leq k-1$, the incident present edge of $v_i$ in $D_1$ is along $C_{i+1}$ but not along $C_i$.
\end{itemize}
\begin{romlist}
\item If the incident present edge of $v_k$ in $D_1$ is along $C_k$, then we rotate configurations of $D_1$ along a sequence of Type-I cycles $C_k',C_{k-1}',\ldots, C_1'$ in order (here for $1\leq i\leq k$, $C_i'$ is a Type-I cycle  enclosing $C_1$), and obtain a new configuration $D_1^{(1)}$, such that
\begin{eqnarray*}
d(D_1^{(1)},D_2)=d(D_1,D_2)-k,
\end{eqnarray*}
because $e_1,\ldots,e_k$ are enclosed in one more cycle in $D_1\Delta D_2$ then $D_1^{(1)}\Delta D_2$.
\item If the incident present edge of $v_k$ in $D_1$ is not along $C_k$, then there is another edge $e_{k+1}=(v_k,v_{k+1})\neq e_k$, such that $e_{k+1}$ is enclosed by $C$. Indeed, let $p_k$ be the incident present edge of $v_k$ in $D_1$.
\begin{problist}
\item If $p_k$ is also present in $D_2$, then all the three incident edges of $v_k$ in $\HH$ are enclosed by $C$. In this case, we choose $e_{k+1}$, such that $e_{k-1},e_k,e_{k+1}$ are along the same face of $\HH$. 
\item If $p_k$ is not present in $D_2$, since $C$ is a contractible cycle, we have $v_k\in C$; because otherwise $C$ will enclose another cycle in $D_1\Delta D_2$ passing through $v_k$. Recall that $S$ forms an angle of size $\frac{4\pi}{3}$ at $v_k$, two incident edges of $v_k$ are in $S$ and enclosed by $C$.
\end{problist}
Since $C$ encloses finitely many edges in total, after finitely many steps, we will have 
\begin{eqnarray}
e_r=e_s,\ \mathrm{for}\ r>s, \label{rgs}
\end{eqnarray}
where $r$ is the least integer such that (\ref{rgs}) occurs. Moreover, no edges of $e_1,\ldots, e_r$ are boundary edges of $\HH_{\Lambda}$, because $C$ does not cross the boundary. The edges $e_s,e_{s+1},\ldots,e_{r}$ form a cycle in $\HH$. If the cycle is not a face of $\HH$, by the induction process described above, we can find a cycle $C'$ in $D_1\Delta D_2$ enclosed by the cycle $e_s,e_{s+1},\ldots,e_{r}$, and therefore $C'$ is also enclosed by $C$. But this is not possible since $C$ is a contractible cycle in $D_1\Delta D_2$. Moreover, the induction process above implies that we must have $s=1$. Therefore $e_1,\ldots,e_6$ are edges of a face $f_1$ in $\HH$. Then we can rotate configurations of $D_1$ to alternate edges along a Type-II cycle enclosing $f_1$, and obtain a new configurations $D_1^{(1)}$ satisfying
\begin{eqnarray*}
d(D_1^{(1)},D_2)=d(D_1,D_2)-6,
\end{eqnarray*}
becuase $e_1,\ldots,e_6$ are enclosed in one more cycle in $D_1\Delta D_2$ than $D_1^{(1)}\Delta D_2$. This completes the proof.
 \end{romlist}
\end{Romlist}
\end{Alist}
\end{proof}

\noindent\textbf{Remark.} A Markov chain with transition matrix defined by (\ref{tmt}) is not necessarily irreducible if its state space is the set of all 1-2 model configurations on a finite subgraph of a hexagonal lattice embedded in a surface other than the Euclidean plane $\RR^2$, with certain boundary conditions. Here is an example. Let $\HH_{C_{n,\infty}}$ be the hexagonal lattice embedded into an infinite cylinder with circumference $n$ and infinite height. Let $\HH_{C_{n,1}}$ be its finite subgraph embedded into a cylinder with circumference $n$ and height $1$. See Figure \ref{fig:cl3}. Consider the boundary condition such that all the edges crossed by $\gamma_1$ are present, and all the edges crossed by $\gamma_2$ are absent. Then the two 1-2 model configurations in Figure \ref{fig:cl3} cannot be obtained by each other by movements described by (2) and (3).

\begin{figure}[htbp]
\includegraphics[width=0.4\textwidth]{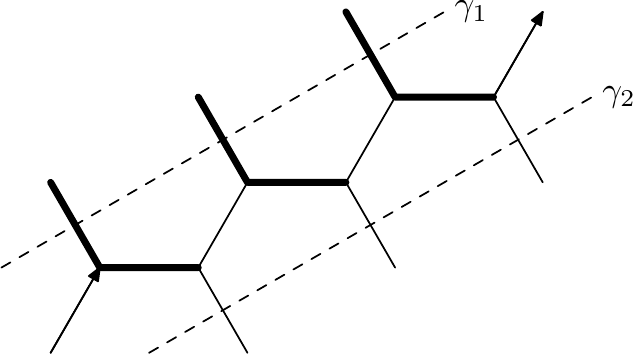}\qquad\qquad\includegraphics[width=0.4\textwidth]{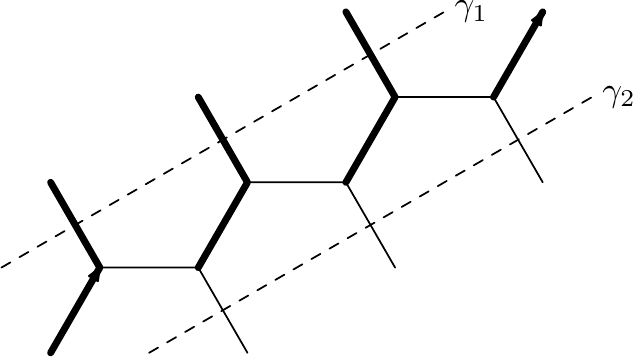}
\caption{1-2 model configurations on cylinder graph. The edges with arrows are identified. Edges crossed by dashed lines $\gamma_1$ and $\gamma_2$ are boundary edges. Thick lines represent present edges in configurations. The configurations on the two graphs cannot be obtained from each other by movements in (2) and (3).}
\label{fig:cl3}
\end{figure}

\begin{proposition}\label{re}Assume the local weights of the 1-2 model are given by $a,b,c>0$; see Figure \ref{fig:sign}. The Markov chain defined by (\ref{tmt}) is a reversible chain, whose stationary distribution is the probability measure on $\Omega_{\Lambda,b_{\Lambda}}$ for which the probability of a configuration is proportional to the product of local weights at each vertex. If $\Lambda=\Lambda_{m,n}$, then the distribution satisfying
\begin{eqnarray}
\pi(\omega)\propto\prod_{v\in \Lambda}w(\omega|_v)\label{gm},\ \omega\in \Omega_{\Lambda,b_{\Lambda}}
\end{eqnarray}
 on $\Omega_{\Lambda,b_{\Lambda}}$ is the unique stationary distribution.
\end{proposition}

\begin{proof}Let $\pi$ be a measure satisfying (\ref{gm}) with local weights given by $a,b,c>0$, and let $P$ be defined by (\ref{tmt}). Then it is straightforward to check that the detailed balance equations hold
\begin{eqnarray*}
\pi(x)P(x,y)=\pi(y)P(y,x),\qquad\forall x,y\in \Omega_{\Lambda,b_{\Lambda}}.
\end{eqnarray*}
This shows that the chain is reversible with the distribution $\pi$ as its stationary distribution. The uniqueness of the stationary distribution follows from the irreducibility, see Proposition \ref{ir}.
\end{proof}

\begin{proposition}\label{ap}The Markov chain defined by (\ref{tmt}) is aperiodic.
\end{proposition}

\begin{proof}By (\ref{tmt}), we infer that for $\omega,\omega'\in \Omega_{\Lambda,b_{\Lambda}}$,  the transition matrix $P$ satisfies
\begin{eqnarray*}
\sum_{\omega'}P(\omega,\omega')=1;\qquad\qquad
P(\omega,\omega')\geq 0;\qquad\qquad
P(\omega,\omega)\geq \frac{1}{2}.
\end{eqnarray*}
The aperiodicity follows from $P(\omega,\omega)>0$, for all  $\omega\in \Omega_{\Lambda,b_{\Lambda}}$.
\end{proof}

\section{The Path Method and the Diameter Bound }\label{sc:ct}

In this section, we review the path method and the comparison theorem which relates the mixing time of the single site dynamics to the mixing time of the block dynamics for irreducible, reversible Markov chains; see Theorems \ref{tm41}-\ref{tm43}. We also discuss a diameter lower bound for the mixing time; see Theorem \ref{tm44}.

Let $\Omega$ be the state space. For a reversible transition matrix $P$, define $E=\{(x,y)\in \Omega^2: P(x,y)>0\}$. Let $x,y\in\Omega$. An $E$-path from $x$ to $y$ is a sequence $\Gamma=(e_1,e_2,\ldots, e_m)$ of edges in $E$, such that $e_1=(x,x_1)$, $e_2=(x_1,x_2)$,\ldots, $e_m=(x_{m-1},y)$ for some vertices $x_1,\ldots x_{m-1}\in \Omega$. The length of an $E$-path $\Gamma$ is denoted by $|\Gamma|$. 

Let $P$ and $\tilde{P}$ be two reversible transition matrices with stationary distributions $\pi$ and $\tilde{\pi}$, respectively.
Let 
\begin{eqnarray}
Q(x,y)&=&\pi(x)P(x,y).\label{q}\\
\tilde{Q}(x,y)&=&\tilde{\pi}(x)\tilde{P}(x,y).\label{tq}
\end{eqnarray}

 Supposing that for each $(x,y)\in \tilde{E}$, there is an $E$-path from $x$ to $y$, choose one and denote it by $\Gamma_{xy}$. Given such a choice of paths, define the \textbf{congestion ratio} $B$ by
\begin{eqnarray}
B:=\max_{e\in E}\left(\frac{1}{Q(e)}\sum_{\{x,y: e\in \Gamma_{xy}\}}\tilde{Q}(x,y)|\Gamma_{xy}|\right)\label{cr}
\end{eqnarray}

\begin{theorem}\label{tm41}(The Comparison Theorem; see Theorem 13.23 of \cite{LPW}.) Let $P$ and $\tilde{P}$ be reversible transition matrices, with stationary distributions $\pi$ and $\tilde{\pi}$, respectively. If $B$ is the congestion ratio for a choice of $E$-paths, as defined in (\ref{cr}), then
\begin{eqnarray*}
\tilde{\gamma}\leq\left[\max_{x\in \Omega}\frac{\pi(x)}{\tilde{\pi}(x)}\right]B\gamma.
\end{eqnarray*}
Here $\gamma=1-\lambda_2$, $\tilde{\gamma}=1-\tilde{\lambda}_2$ are the spectral gaps for $P$, $\tilde{P}$, respectively, and $\lambda_2$, $\tilde{\lambda}_2$ are second largest eigenvalues for $P$, $\tilde{P}$, respectively. 
\end{theorem}

\begin{theorem}\label{tm42}(Theorem 12.3 of \cite{LPW}) Let $P$ be the transition matrix of a reversible, irreducible Markov chain with state space $\Omega$, and let $\pi_{min}=\min_{x\in\Omega}\pi(x)$, then
\begin{eqnarray*}
t_{mix}(\epsilon)\leq\log\left(\frac{1}{\epsilon\pi_{\min}}\right)t_{rel},
\end{eqnarray*}
where $t_{rel}$ is the relaxation time given by 
\begin{eqnarray*}
t_{rel}=\frac{1}{\gamma_*},
\end{eqnarray*}
and $\gamma_*$ is the absolute spectral gap defined by
\begin{eqnarray*}
\gamma_*&=&1-\lambda_*\\
\lambda_*&=&\max\{|\lambda|: \lambda\ \mathrm{is\ an\ eigenvalue\ of\ }P, \lambda\neq 1 \}.
\end{eqnarray*}
\end{theorem}

Exercise 12.3 of \cite{LPW} shows that if a chain is lazy, then all the eigenvalues of its transition matrix is nonnegative, and therefore $\gamma_*=\gamma$.

\begin{theorem}(Theorem 12.4 of \cite{LPW})\label{tm43} For a reversible, irreducible, and aperiodic Markov chain
\begin{eqnarray*}
t_{mix}\geq (t_{rel}-1)\log \left(\frac{1}{2\epsilon}\right).
\end{eqnarray*}
\end{theorem}

\begin{theorem}\label{tm45}Let $P$ be a transition matrix on a metric space $(\Omega,\rho)$, where the metric $\rho$ satisfies $\rho(x,y)\geq \mathbf{1}\{x\neq y\}$. Suppose, for all states $x$ and $y$, there exists a coupling $(X_1,Y_1)$ of $P(x,\cdot)$ with $P(y,\cdot)$ that contracts $\rho$ on average, i.e., which satisfies
\begin{eqnarray*}
\mathbf{E}_{x,y}\rho(X_1,Y_1)\leq e^{-\alpha}\rho(x,y),
\end{eqnarray*}
for some $\alpha>0$. Let
\begin{eqnarray*}
D_{\rho}(\Omega):=\max_{x,y\in \Omega}\rho(x,y),
\end{eqnarray*}
then we have
\begin{eqnarray*}
t_{mix}(\epsilon)\leq\left\lceil\frac{\log(D_{\rho}(\Omega))+\log\left(\frac{1}{\epsilon}\right)}{\alpha}\right\rceil
\end{eqnarray*}
\end{theorem}
\begin{proof}See Section 14 of \cite{LPW}.
\end{proof}

Next we discuss a diameter lower bound on the mixing time. Given a transition matrix $P$ on $\Omega$, construct a graph with vertex set $\Omega$ and which includes the edge $\{x,y\}$ for all $x$ and $y$ with $P(x,y)+P(y,x)>0$. Define the \textbf{diameter} of a Markov chain to be the diameter of this graph, that is, the maximal graph distance between distinct vertices.

\begin{theorem}\label{tm44}Let $P$ be an irreducible and aperiodic transition matrix with diameter $L$, then for any $\epsilon<\frac{1}{2}$
\begin{eqnarray*}
t_{mix}(\epsilon)\geq \frac{L}{2}.
\end{eqnarray*}
\end{theorem}
\begin{proof}See Section 7.1.2 of \cite{LPW}.
\end{proof}

\section{Mixing Time for the 1-2 Model on  a $k\times n$ Rectangle}\label{sc:cl}

In this section, we prove Theorem \ref{m1}.

Let $\Lambda_{k,n}$, $b_{k,n}$ be given as in Theorem \ref{m1}. Let $F$ be the set of faces of $\Lambda_{k,n}$. Assume that the vertices in $\Lambda_{k,n}$ are labeled by ordered pairs of integers $(p,q)$ satisfying $1\leq p\leq k$ and $1\leq q\leq n$. See Figure \ref{fig:cld} for an example.

\begin{figure}[htbp]
\centerline{\includegraphics*{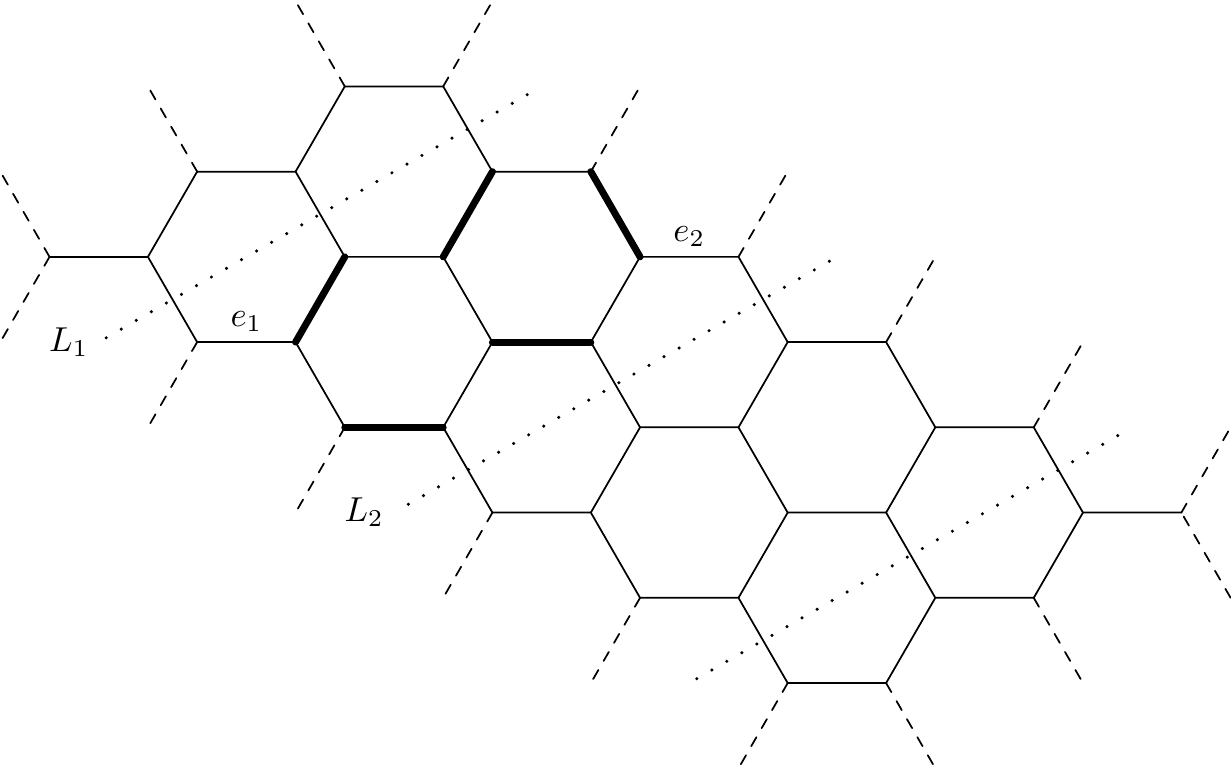}}
   \caption{A rectangular hexagonal lattice $\Lambda_{3,6}$. Dashed edges are boundary edges; dotted lines are boundaries separating different $G_i$'s.}
   \label{fig:cld}
\end{figure}

Assume that both $n$ and $\ell$ are even.
For $1\leq j\leq \frac{\ell}{2}$, define the $j$th block to be the vertex set
\begin{eqnarray*}
V_j=\{(p,q):1\leq p\leq k, 1\leq q\leq 2j\}.
\end{eqnarray*}
For $\frac{\ell}{2}+1\leq j\leq n-\frac{\ell}{2}$, define the $j$th block to be the vertex set 
\begin{eqnarray*}
V_j=\{(p,q):1\leq p\leq k,  j-\frac{\ell}{2}+1\leq p\leq j+\frac{\ell}{2}\}.
\end{eqnarray*}
For $n-\frac{\ell}{2}+1\leq j\leq n$, define the $j$th block to be the vertex set 
\begin{eqnarray*}
V_j=\{(p,q):1\leq p\leq k, 2j-n-1\leq q\leq n\}.
\end{eqnarray*}
Let $E_j$ be the set of edges whose both endpoints are in $V_j$.

The block dynamics for the 1-2 model on $\Lambda_{k,n}$ is the Markov chain defined as follows. A block $V_i$ is picked uniformly at random among the $n$ blocks, and the configuration $\sigma$ is updated according to the measure $\pi$ on 1-2 model configurations satisfying (\ref{gm}) and  conditioned to agree with $\sigma$ everywhere outside $E_i$. More precisely, let $\Omega_{\Lambda_{k,n},b_{k,n}}\subset\{0,1\}^E$ be the set of all 1-2 model configurations on $\Lambda_{k,n}$ with the admissible boundary condition $b_{k,n}$, and let
\begin{eqnarray*}
\Omega_{\sigma,E_i}=\{\tau\in\Omega_{\Lambda_{k,n},b_{k,n}}:\tau(e)=\sigma(e)\qquad\mathrm{for\ all\ }e\notin E_i\}
\end{eqnarray*}
be the set of configurations agreeing with $\sigma$ outside $E_i$. Define the transition matrix
\begin{eqnarray*}
P_{E_i}(\sigma,\tau)=\pi(\tau|\Omega_{\sigma,E_i}).
\end{eqnarray*}

The block dynamics has transition matrix
\begin{eqnarray}
\tilde{P}:=\frac{1}{n}\sum_{i=1}^{n}P_{E_i}.\label{bd}
\end{eqnarray}

It is straightforward to check the following lemma.

\begin{lemma}The block dynamics with transition matrix defined by (\ref{bd}) is an irreducible, reversible, aperiodic Markov chain whose stationary distribution is the  distribution $\pi$ satisfying (\ref{gm}) on $\Omega_{\Lambda_{k,n},b_{k,n}}$.
\end{lemma}

We will apply the comparison theorem, which requires that we define for each block move from $\sigma$ to $\tau$, a sequence of single site moves from $\sigma$ and ending at $\tau$. 

Let $\sigma,\tau\in\Omega_{\Lambda_{k,n},b_{k,n}}$, such that $\sigma$ and $\tau$ differ only in $E_i$. Let $\HH_{\Delta,\Lambda_{k,n}}$ be a $k\times n$ rectangular subgraph  the decorated graph $\HH_{\Delta}$ corresponding to $\Lambda_{k,n}$. Let $D_{\sigma}$ (resp.\ $D_{\tau}$) be the dimer configuration on $\HH_{\Delta, \Lambda_{k,n}}$ corresponding to $\sigma$ (resp.\ $\tau$). Assume that $\ell\gg k$. We have
\begin{eqnarray}
d(D_{\sigma},D_{\tau})\leq C_1(k)\ell,\qquad \forall \sigma,\tau\in \Omega.\label{dst}
\end{eqnarray}
where $C_1(k)$ is a constant depending only on $k$, and $d(D_{\sigma},D_{\tau})$ is the distance of $D_{\sigma}$ and $D_{\tau}$ as defined in (\ref{d12}). Using the process as in the proof of Proposition \ref{ir}, we find a path
\begin{eqnarray*}
\sigma(=\sigma_0),\sigma_1,\ldots,\sigma_k(=\tau),
\end{eqnarray*}
such that $\sigma_i$ and $\sigma_{i+1}$ ($0\leq i\leq k-1$) differ at exactly one edge or one face, as described in (2) or (3). Note that
\begin{eqnarray*}
|\Gamma_{\sigma\tau}|\leq C_1(k)\ell,
\end{eqnarray*}
by (\ref{dst}). For these paths we must bound the congestion ratio defined by (\ref{cr}).

Let $\xi=(\sigma_0,\tau_0)$, where $\sigma_0$ and $\tau_0$ agree everywhere except at an edge $e$ or a face $f$, as described in (2) or (3).  Let $Q(\xi)$, $\tilde{Q}(\xi)$ be defined as in (\ref{q}), (\ref{tq}), respectively.

We have
\begin{eqnarray*}
R_{\xi}:&=&\frac{1}{Q(\xi)}\sum_{\{\sigma,\tau:\xi\in \Gamma_{\sigma,\tau}\}}\pi(\sigma)\tilde{P}(\sigma,\tau)|\Gamma_{\sigma,\tau}|\\
&\leq &C_1(k) \ell\sum_{\{\sigma,\tau:\xi\in \Gamma_{\sigma,\tau}\}}\frac{1}{n}\sum_{i=1}^n\frac{\tilde{P}_{E_i}(\sigma,\tau)\pi(\sigma)}{P(\sigma_0,\tau_0)\pi(\sigma_0)}.
\end{eqnarray*}
where $\pi$ is the distribution on $\Omega_{\Lambda_{k,n},b_{k,n}}$ in which the probability of a configuration in proportional to the product of local weights at each vertex, and the local weights are given by $a,b,c>0$. Since $\sigma$ and $\sigma_0$ differ in at most one block, we have
\begin{eqnarray*}
\left(\frac{\min\{a,b,c\}}{\max\{a,b,c\}}\right)^{2lk}\leq \frac{\pi(\sigma)}{\pi(\sigma_0)}\leq \left(\frac{\max\{a,b,c\}}{\min\{a,b,c\}}\right)^{2lk}
\end{eqnarray*}
 Let 
\begin{eqnarray*}
M&=&\max_{1\leq i\leq n}|E_i|\leq C_2(k)\ell\\
M^*&=&\max\{\max_{e\in E_{n,k}}|\{i,e\in E_i\}|,\max_{f\in F}|\{i: f\subseteq E_i\}\}|\leq C_2(k)\ell,
\end{eqnarray*}
where $C_2(k)>0$ is a constant depending only on $k$; by $f\subseteq E_i$ we mean every edge of the face $f$ is in $E_i$. We have
\begin{eqnarray*}
\frac{P_{E_i}(\sigma,\tau)}{P(\sigma_0,\tau_0)}=\frac{2(|E_{\Lambda_{n,k}}|+|F_{\Lambda_{n,k}}|)[\max\{a,b,c\}]^2}{[\min\{a,b,c\}]^2}\pi(\tau|\Omega_{\sigma,E_i}),
\end{eqnarray*}
where $P(\sigma_0,\tau_0)$ is defined by (\ref{tmt}), $|E_{\Lambda_{n,k}}|$ (resp.\ $|F_{\Lambda_{n,k}}|$) is the number of edges (resp.\ faces) in $\Lambda_{n,k}$.
Hence we have
\begin{eqnarray}
R_{\xi}&\leq& \frac{C_1(k)\ell}{n}\sum_{\{\sigma,\tau:\xi\in\Gamma_{\sigma,\tau}\}}\sum_{i=1}^{n}2(|E_{\Lambda_{n,k}}|+|F_{{\Lambda}_{n,k}}|)\pi(\tau|\Omega_{\sigma,E_i})\left(\frac{\max\{a,b,c\}}{\min\{a,b,c\}}\right)^{2lk+1}\notag\\
&\leq&\frac{2C_1(k)\ell(|E_{\Lambda_{n,k}}|+|F_{\Lambda_{n,k}}|)}{n}\left(\frac{\max\{a,b,c\}}{\min\{a,b,c\}}\right)^{2lk+1}\left[\sum_i\mathbf{1}_{\{e\in E_i\ \text{or}\ f\subseteq E_i\}}\right.\\
&&\times\left.\left(\sum_{\sigma,\tau:\xi\in \Gamma_{\sigma\tau}}\pi(\tau|\Omega_{\sigma,E_i})\right)\right]\notag\\
&\leq&C_3(k)\ell^2 C_5^{C_4(k)\ell},\label{rx}
\end{eqnarray}
where $C_3(k)$, $C_4(k)$ are constants depending only on $k$, and $C_5$ is a constant depending on $a,b,c$ and independent of $k$.

Thus we have the following lemma
\begin{lemma}\label{l52}Let $\gamma_B$ (resp. $\gamma$) be the spectral gap for the block (single-edge or single face) dynamics of the 1-2 model on the $n\times k$ box $\Lambda_{n,k}$ of the hexagonal lattice $\HH$ with transition matrix $\tilde P$ (resp.\ $P$) defined by (\ref{bd}) (resp.\ (\ref{tmt})), then
\begin{eqnarray*}
\gamma_B\leq C_3(k)\ell^2 C_5^{C_4(k)\ell}\gamma,
\end{eqnarray*}
where $C_3(k)$, $C_4(k)$ are constants depending only on $k$, and $C_5$ is a constant depending on $a,b,c$ and independent of $k$.
\end{lemma}

\begin{proof}The lemma follows from Theorem \ref{tm41} and (\ref{rx}).
\end{proof}

Recall that for the Markov chain with transition matrix defined by (\ref{tmt}), the spectral gap is the same as the absolute spectral gap; see Proposition \ref{ap}. Let $t_{mix}$ (resp.\ $t_{rel}$) be the mixing time (resp.\ relaxation time) for the Markov chain with transition matrix defined by (\ref{tmt}) and state space $\Omega_{\Lambda_{k,n},b_{k,n}}$.
Then by Lemma \ref{l52} and Theorem \ref{tm42}, we have
\begin{eqnarray}
t_{mix}(\epsilon)&\leq& \log\left(\frac{1}{\epsilon\min_{\omega\in \Omega_{\Lambda_{k,n},b_{k,n}}}\pi(\omega)}\right)t_{rel}\\
&\leq&  C_3(k)\ell^2 4^{C_4(k)\ell}\log\left(\frac{1}{\epsilon\min_{\omega\in \Omega_{\Lambda_{k,n},b_{k,n}}}\pi(\omega)}\right)\frac{1}{\gamma_B},\label{tm}
\end{eqnarray}
where $\pi(\cdot)$, as before, is a probability measure on $\Omega_{\Lambda_{k,n},b_{k,n}}$, where the probability of a 1-2 model configuration is proportional to the product of local weight on each vertex.

If a block $V_j$ is selected, we will apply a sequential method of updating as follows. When $\ell$ is even, we divide the block $V_j$ into $\frac{\ell}{2}$ different sub-blocks: $G_{j,1},\ldots, G_{j,\frac{\ell}{2}}$, such that each $G_{j,i} (1\leq i\leq \frac{\ell}{2})$ is of size $2\times k$, see Figure \ref{fig:cld}, where different $2\times k$ columns are separated by dotted lines. If the block has less than $\ell$ columns, we still divide the block into sub-blocks, each of which has size $2\times k$.

For $1\leq i<\frac{\ell}{2}$, let $\gamma_{j,i}$ be the random variable denoting the states of all edges in $G_{j,i}$, as well as the states of all edges connecting $G_{j,i}$ and $G_{j,i+1}$. Let $\gamma_{j,\frac{\ell}{2}}$ be the random variable denoting the states of all the edges in $G_{j,\frac{\ell}{2}}$, as well as the states of all the edges connecting $G_{j,\frac{\ell}{2}}$ and $G_{j+\ell,1}$. The conditional distribution of $\gamma_{j,i+1}$, given $(\gamma_{j,1},\ldots, \gamma_{j,i})$ and $\gamma_{j+\ell,1}$ depends only on $\gamma_{j,i}$ and $\gamma_{j+\ell,1}$. Therefore given $\gamma_{j-\ell,\frac{\ell}{2}}$ and $\gamma_{j+\ell,1}$, the sequence $(\gamma_{j,i})_{i=1}^{\frac{\ell}{2}}$ is a time-inhomogeneous Markov chain. If a block $V_j$ is selected to be updated in the block dynamics, the update can be realized by running this chain.
 
We now describe how to couple the block dynamics started from $\sigma$, with the block dynamics started from $\tau$, in the case that $\sigma$, $\tau$ differ at only one edge $e$ or only one face $f$, as described in (2) and (3). Always select the same block to update the two chains. Once a block $V_j$ is selected, the block $V_j$ is updated sub-block by sub-block from the first sub-block $G_{j,1}$ to the last sub-block $G_{j,\frac{\ell}{2}}$. When updating the sub-block $G_{j,1}$, it is updated from the conditional distribution given the configurations in $G_{j-\ell,\frac{\ell}{2}}$ and $G_{j+\ell,1}$. For $2\leq i\leq \frac{\ell}{2}$, the sub-block $G_{j,i}$ is updated from the conditional distribution given the configurations $G_{j,i-1}$ and $G_{j+\ell,1}$. If in $\sigma$ and $\tau$, the configurations $G_{j,i-1}$ and $G_{j+\ell,1}$ are the same, then in the sub-block $G_{j,i}$, the two chains are updated together, i.e. after the update, in the sub-block $G_{j,i}$, the two chains have the same configuration.

 If a block is selected which contains the edge $e$ or the face $f$, then the two chains can be updated together, and the difference of $\sigma$ and $\tau$ can be eliminated. The difficulty occurs when
\begin{Alist}
\item $e$ is incident to exactly one vertex in the selected block; or
\item at least one vertex in $f$ is in the selected block and at least one vertex in $f$ is outside the selected block.
\end{Alist}

We consider Case A. first. Assume that a block $V_j$ is selected. We will update the sub-blocks starting from the sub-block that is incident to the edge $e$, and label the sub-block by $G_{j,1}$. We make the following claim
\begin{claim}\label{c53}
 Whatever the states of the edges connecting $G_{j,i-1}$ and $G_{j,i}$ and the edges connecting $G_{j,i}$ and $G_{j,i+1}$ are, there is always a strictly positive probability that on all the edges connecting two points of $G_{j,i}$, $\sigma$ and $\tilde{\sigma}$ have the same states. 
\end{claim}

To see why Claim \ref{c53} is true, consider the sub-block in Figure \ref{fig:cld} bounded by $L_1$ and $L_2$. An interior edge of the block is an edge whose both endpoints are in the sub-block. It is not hard to see that the specific configuration on all the interior edges of the sub-block except for $e_1$ and $e_2$ as illustrated in Figure \ref{fig:cld} can be extended to a 1-2 model configuration with any boundary configurations of the sub-block. Therefore whatever the states of the edges crossing $L_1$ and $L_2$ are, there is always a strictly positive probability that on all the interior edges except $e_1$ and $e_2$  have the same states. The interior edges $e_1$ and $e_2$ are special since both of them has an endpoint incident to two boundary edges of the sub-block. If with probability 1, $e_1$ (resp. $e_2$) has different states in $\sigma$ and $\tilde{\sigma}$, then both boundary edges incident to $e_1$ (resp. $e_2$) are present in $\sigma$ and absent in $\tilde{\sigma}$; or both boundary edges incident to $e_1$ (resp.\ $e_2$) are absent in $\sigma$ but present in $\tilde{\sigma}$. But this is impossible since by assumption $\sigma$ and $\tilde{\sigma}$ have the same state on the boundary edge of $\Lambda_{k,n}$ which is also incident to $e_1$.

Let $C_5(k)>0$
 be a lower bound for the probability that all the edges in a sub-block have the same states in $\sigma$ and $\tilde{\sigma}$; where the lower bound is taken over all the possible states of edges connecting $G_{j,i-1}$ and $G_{j,i}$, and edges connecting $G_{j,i}$ and $G_{j,i+1}$. 
 
Now we compute the expected number of vertices in the block $V_j$ where the two updates disagree. For $1\leq i\leq\frac{\ell}{2}$, if there exists a difference of configurations between the two updates in the $i$th sub-block $G_{j,i}$, then there is difference between the two updates in each one of the 1st, 2nd, ..., and $(i-1)$th sub-block updated before the update of $G_{j,i}$. Hence we have
\begin{eqnarray*}
\PP(\#\ \mathrm{of\ vertices\ in\ }V_j\ \mathrm{such\ that\ the\ two\ updates\ disagree}\geq ik)\leq (1-C_5(k))^i.
\end{eqnarray*}
 Therefore the expected number of vertices in the block $V_j$ where the two updates disagree  is bounded by $C_6(k)$, where $C_6(k)>0$ is a constant depending only on $k$, but independent of $\ell$.
 
 Let $\rho(\sigma,\tau)$ be the number of edges with different states in $\sigma$ and $\tau$. First let us treat the case when $\sigma$ and $\tau$ differ at exactly one edge $e$. Then $\rho(\sigma,\tau)=1$. Let $(X_1,Y_1)$ be the pair of configurations obtained after one step of coupling. Since $\ell$ of the $n$-blocks will contain the edge $e$; and two of the blocks contain exactly one endpoint of $e$, we have
 \begin{eqnarray*}
 \mathbb{E}_{\sigma,\tau}\rho(X_1,Y_1)\leq 1-\frac{\ell}{2n}+\frac{2C_6(k)}{n}.
 \end{eqnarray*}
 If we choose $\ell=4C_6(k)+2$, then
 \begin{eqnarray*}
 \mathbb{E}_{\sigma,\tau}\rho(X_1,Y_1)\leq 1-\frac{1}{n}\leq e^{-\frac{1}{n}},
 \end{eqnarray*}
 for $\sigma,\tau$ with $\rho(\sigma,\tau)=1$.
 
 Now let us treat the case when $\sigma$ and $\tau$ differ at exactly one face, as described in (3) of Section \ref{mch}; see also Figure \ref{fig:p3}. Now $\rho(\sigma,\tau)=6$. There are at least $\ell-6$ blocks containing every vertex of the face, and at most 6 blocks containing some vertices of the face, but not all the vertices of the face. Using the same arguments as above,  we have
 \begin{eqnarray*}
  \mathbb{E}_{\sigma,\tau}\rho(X_1,Y_1)\leq 1-\frac{\ell-6}{2n}+\frac{6C_6(k)}{n}.
\end{eqnarray*} 
Therefore by choosing $\ell$ sufficiently large, we obtain
 \begin{eqnarray*}
 \mathbb{E}_{\sigma,\tau}\rho(X_1,Y_1)\leq e^{-\frac{1}{n}}\rho(\sigma,\tau).
 \end{eqnarray*}

Recall the following theorem

\begin{theorem}\label{tm53}(Theorem 13.1 of \cite{LPW}; see also \cite{Mch}) Let $\Omega$ be a metric space with metric $\rho$, and let $P$ be the transition matrix of a Markov chain with state space $\Omega$. Suppose there exists a constant $\theta<1$ such that for each $x,y\in\Omega$, there exists a coupling $(X_1,Y_1)$ of $P(x,\cdot)$ and $P(y,\cdot)$ satisfying
\begin{eqnarray}
\mathbb{E}_{x,y}(\rho(X_1,Y_1))\leq \theta \rho(x,y).\label{ctr}
\end{eqnarray}
If $\lambda\neq 1$ is an eigenvalue of $P$, then $|\lambda|\leq \theta$. In particular, the absolute spectral gap satisfies
\begin{eqnarray*}
\gamma_*\geq 1-\theta.
\end{eqnarray*}
\end{theorem}

By (\ref{ctr}), and Theorem \ref{tm53}, we have
\begin{eqnarray}
\gamma_B\geq \gamma_{B,*}\geq 1-e^{\frac{1}{-n}}\geq \frac{1}{n},\label{s1}
\end{eqnarray}
where $\gamma_B$ is the spectral gap for the block dynamics, and $\gamma_{B,*}$ is the absolute spectral gap for the block dynamics.
Since
\begin{eqnarray}
\log\left(\frac{1}{\epsilon\min_{\omega\in \Omega_{\Lambda_{k,n},b_{k,n}}}\pi(\omega)}\right)\leq C(k,\epsilon)n,\label{s2}
\end{eqnarray}
By  (\ref{tm}), (\ref{s1}), (\ref{s2}), we have
\begin{eqnarray*}
t_{mix}\leq C_7(k)n^2,
\end{eqnarray*}
for the Markov chain described by (\ref{tmt}). Moreover, by Theorem \ref{tm44}, we have
\begin{eqnarray*}
t_{mix}\geq B(k) n,
\end{eqnarray*}
where $B(k)$ is a constant depending on $k$ but is independent of $n$. This completes the proof of Theorem \ref{m1}.

We can also consider the mixing time of the block dynamics with transition matrix given by (\ref{bd}). The following theorem holds.

\begin{theorem}\label{tm54}Let $\tilde{P}$ defined by (\ref{bd}) be the transition matrix for the block dynamics of 1-2 model configurations on $\Omega_{\Lambda_{k,n},b_{k,n}}$, then
\begin{eqnarray*}
B'(k) n \leq \tilde{t}_{mix}\leq C'(k) n\log n,
\end{eqnarray*}
where $B'(k)>0$, $C'(k)>0$ are positive constants depending only on $k$ and independent of $n$, and $\tilde{t}_{mix}$ is defined by (\ref{mtm}) with respect to $\tilde{P}$.
\end{theorem}
\begin{proof}The upper bound $B'(k)n$ is obtained following Theorem \ref{tm45}, and the lower bound is obtained following Theorem \ref{tm44}.
\end{proof}

\section{Mixing Time for the 1-2 Model on an $n\times n$ Box}\label{sc:wm}

In this section, we prove Theorem \ref{m2}. We first prove the spatial mixing of the Gibbs measures for the uniform 1-2 model in Section \ref{ssc:sm}, then we introduce the block dynamics in Section \ref{ssc:bd}. The spatial mixing of the Gibbs measures implies the fast mixing of the block dynamics; the comparison theorem and the fast mixing of the block dynamics imply the fast mixing of single site dynamics.

\subsection{Spatial mixing of Gibbs measures for the 1-2 model}\label{ssc:sm}

Let $\Lambda=(V_{\Lambda},E_{\Lambda})$ be a finite subgraph of the hexagonal lattice $\HH$. A 1-2 model configuration on $\Lambda$ can be considered as a spin system $\{\sigma_e\}_{e\in E_{\Lambda}}$,
\begin{enumerate}
\item for each $e\in E_{\Lambda}$, $\sigma_e\in\{\pm1\}$;
\item let $e_1,e_2,e_3$ be the three incident edges of a vertex $v$, then
\begin{eqnarray*}
\sigma_{e_1}+\sigma_{e_2}+\sigma_{e_3}\in\{\pm1\}.
\end{eqnarray*}
\end{enumerate}
Indeed, an edge $e$ is present in a 1-2 model configuration if and only if $\sigma_e=1$. Condition (2) ensures that each vertex has one or two incident present edges in the 1-2 model configuration.

A Kagome lattice ([3,6,3,6] lattice) can be constructed from the hexagonal lattice $\HH$ as follows. Place a vertex of the Kagome lattice at the center of each edge of $\HH$; two vertices of the Kagome lattice are adjacent, or joined by an edge of the Kagome lattice, if and only if they are at the centers of two edges of $\HH$ sharing a vertex. See Figure \ref{fig:kag}. 
We can consider the 1-2 model on $\HH$ as a spin system with spins located on vertices of the Kagome lattice with nearest neighbor interactions. In particular the vertex set of a Kagome lattice is in one-to-one correspondence with the edge set of the hexagonal lattice.
\begin{figure}[htbp]
\includegraphics{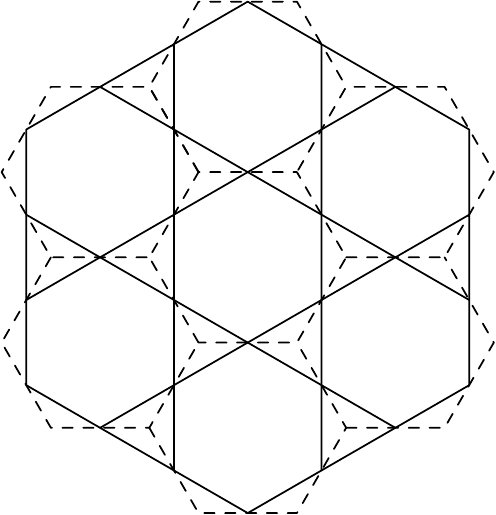}
\caption{The Kagome lattice ([3,6,3,6] lattice) and the hexagonal lattice. The Kagome lattice is represented by black lines, while the hexagonal lattice is represented by dashed lines.}
\label{fig:kag}
\end{figure}

Let $e_1,e_2,e_3$ be three incident edges of a vertex $v$ of $\HH$. The potential is given by
\begin{eqnarray}
U_{e_1,e_2,e_3}(\sigma_{e_1},\sigma_{e_2},\sigma_{e_3})=-\log\left(1+A\sigma_{e_1}\sigma_{e_2}+B\sigma_{e_1}\sigma_{e_3}+C\sigma_{e_2}\sigma_{e_3}\right),\label{pu}
\end{eqnarray}
where 
\begin{eqnarray*}
A=\frac{a-b-c}{a+b+c},\qquad B=\frac{b-a-c}{a+b+c},\qquad C=\frac{c-a-b}{a+b+c}
\end{eqnarray*}
Such a correspondence between 1-2 model on $\HH$ with local weights $a,b,c$ and spin system was introduced in \cite{GL1}; see also \cite{GL2}.

The Hamiltonian for 1-2 model configurations on $\Lambda$, with boundary condition $\tau$, is defined by 
\begin{eqnarray*}
H_{\Lambda}^{\tau}(\sigma)=\sum_{v\in V_{\Lambda}}U_{Dv}(\sigma_{e_1^v},\sigma_{e_2^v},\sigma_{e_3^v}),
\end{eqnarray*}
where $Dv=\{e_1^v,e_2^v,e_3^v\}$ is the set consisting of three incident edges of $v$ in the hexagonal lattice $\HH$, and the spins on edges joining one vertex in $\Lambda$ and one vertex outside $\Lambda$ have states given by the boundary condition $\tau$.

A finite-volume Gibbs measure on $\Lambda$ with boundary condition $\tau$ is the probability distribution on $\{0,1\}^{E_n}$ defined by
\begin{eqnarray}
\mu_{\Lambda}^{\tau}(\sigma)=\frac{e^{-H_{\Lambda}^{\tau}(\sigma)}}{Z},\label{gb12}
\end{eqnarray}
where $Z$ is a normalizing constant called the partition function,  which depends on the boundary condition $\tau$. It is proved in \cite{GL1} that the measure defined by (\ref{gb12}) is exactly the probability measure on 1-2 model configurations under which the probability of a configuration is proportional to the product of weights of local configurations.

Let $\Delta$ be a finite subgraph of $\Lambda$. Let $\mu_{\Lambda, \Delta}^{\tau}$ be the marginal distribution of $\mu_{\Lambda}^{\tau}$ restricted on events depending only on states of edges in $\Delta$.

\begin{definition}Let $\Lambda$ be a finite subgraph of $\HH$, and let $\tau,\tau'$ be two admissible boundary conditions for 1-2 model configurations on $\Lambda$ such that $\tau$ and $\tau'$ differ at a finite set of edges $E_{\tau\tau'}$ satisfying $|E_{\tau\tau'}|\leq 6$. Gibbs measures for 1-2 model configurations on $\Lambda$ are strong mixing if for any $\Delta\subset \Lambda$,
\begin{eqnarray}
\max^*_{\tau,\tau'}\|\mu_{\Lambda,\Delta}^{\tau}-\mu_{\Lambda,\Delta}^{\tau'}\|_{TV}\leq C e^{-\gamma \mathrm{dist}(\Delta,E_{\tau\tau'})},\label{smc}
\end{eqnarray}
where $C,\gamma>0$ are constants independent of $\Delta$, $\|\cdot\|_{TV}$ is the total variation distance of two probability measures, and the superscript $*$ denotes that we maximize over those pairs $\tau,\tau'$ which agree off $E_{\tau\tau'}$, i.e. which satisfy $\tau(e)=\tau'(e)$, for $e\in \partial\Lambda\setminus E_{\tau\tau'}$.
\end{definition}

The strong mixing condition for the 1-2 model is defined analogously to the strong mixing condition for a general (unconstrained) spin system; see also \cite{mos94,MF99}.

\begin{definition}\label{d62}Let $n\in\NN$ and $\epsilon>0$. Let $\Lambda_n$ be the $n\times n$ box of the hexagonal lattice $\HH$ centered at the origin. Let $B(n)$ be the set consisting of all the admissible boundary conditions for $\Lambda_n$. Let $F(n,\epsilon)$ be the condition that
\begin{eqnarray*}
&&\max_{\tau,\tau'\in B(3n)}\|\mu_{\Lambda_{3n},\Lambda_n}^{\tau}-\mu_{\Lambda_{3n},\Lambda_n}^{\tau'}\|_{TV}+2\max^*_{\tau,\tau'\in B(n)}\|\mu_{\Lambda_n,M_1(\Lambda_n)}^{\tau}-\mu_{\Lambda_n, M_1(\Lambda_n)}^{\tau'}\|_{TV}\\&&+2\max^{**}_{\tau,\tau'\in B(n)}\|\mu_{\Lambda_n,M_2(\Lambda_n)}^{\tau}-\mu_{\Lambda_n, M_2(\Lambda_n)}^{\tau'}\|_{TV}<\epsilon,
\end{eqnarray*}
where $*$ (resp.\ $**$) denotes that the maximum is taken over those pairs $\tau,\tau'$ with $\tau=\tau'$ on the NW and SE boundaries (resp.\ NE and SW) of $\Lambda_n$, and $M_1(\Lambda_n)$ (resp.\ $M_2(\Lambda_n)$) is the central line of $\Lambda_n$ parallel the the NE (resp.\ NW) boundary.
\end{definition}

\begin{figure}[htbp]
\includegraphics[width=0.5\textwidth]{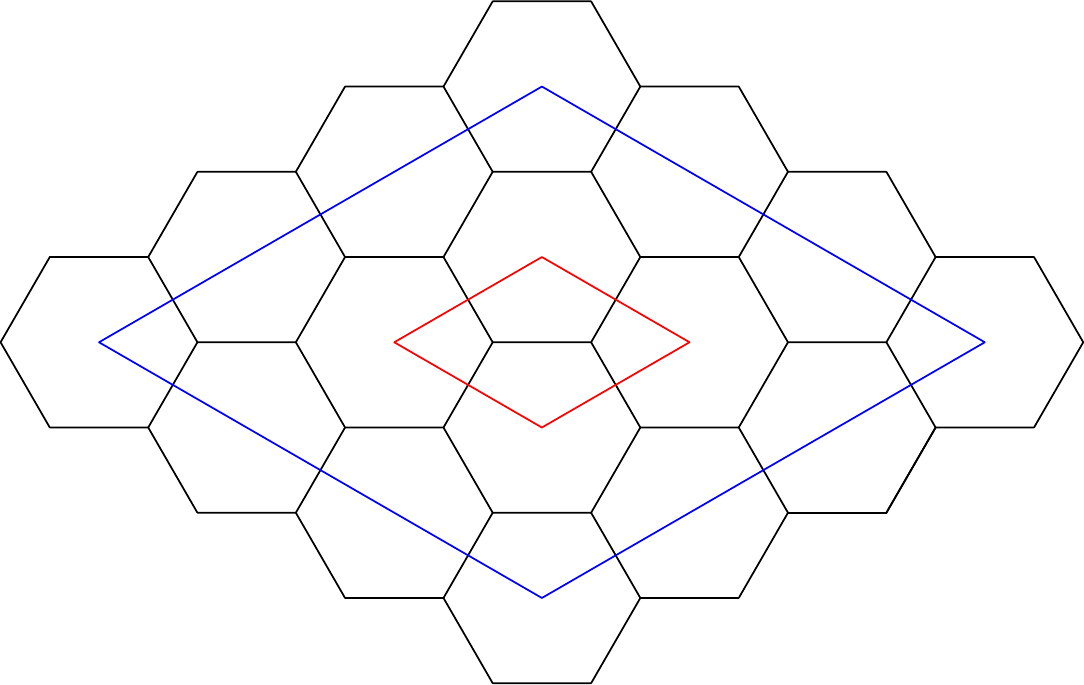}
\caption{A $\Lambda_n$ box in the center of a $\Lambda_{3n}$ box when $n=1$. The $\Lambda_1$ box is bounded by the red line; and the $\Lambda_{3}$ box is bounded by the red line. The boundary condition of $\Lambda_{3}$ box is given by the configurations on all the vertices outside $\Lambda_3$.}
\label{fig:lb13}
\end{figure}

\begin{figure}[htbp]
\includegraphics[width=0.5\textwidth]{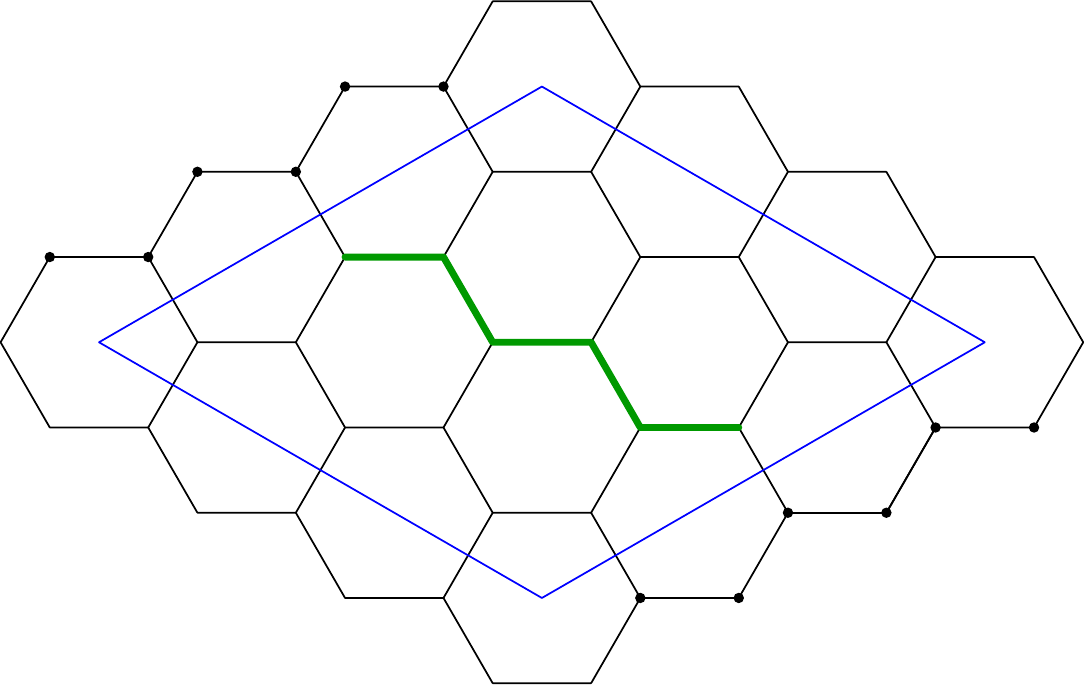}
\caption{The green line represents the central line $M_1(\Lambda_3)$ parallel to the northeastern boundary. The labeled vertices are vertices with the same configurations for $\tau$ and $\tau'$ when computing $\max^{*}_{\tau,\tau'}$.}
\label{fig:cl3}
\end{figure}

\begin{theorem}\label{t34}Let $G$ be a connected component of the graph obtained from the square grid by joining each vertex $(p,q)\in \ZZ^2$ with 16 vertices $(p,q\pm 2),(p\pm1,q\pm 2),(p\pm 2,q\pm 2),(p\pm 2,q),(p\pm 2,q\pm 1)$. Let $\nu_k$ be the number of $k$-step self-avoiding walks on $G$ starting from $(0,0)$.
Let $\nu$ be the connective constant for $G$ defined by $\nu=\lim_{k\rightarrow\infty}\nu_k^{\frac{1}{k}}$. If there exists $N$, such that $F(N,\nu^{-1})$ holds, then the strong mixing condition holds for Gibbs measures on sufficiently large box $\Lambda_n$ with constants $C,\gamma$ independent of $n$.
\end{theorem}

\begin{proof}The theorem can be proved using the same technique as in the proof of Proposition 1(a) of \cite{Jvdb}. The major difference between the our setting and the setting in \cite{Jvdb} is that our potential function defined by (\ref{pu}) may take the value $\infty$, due to the constraint that one or two incident edges are present at each vertex; and we are working on a 2D Kagome lattice instead of a square grid. However, the techniques used in \cite{Jvdb} work in our setting as well. We briefly sketch the proof here. 

Let $\Lambda_n$ be an $n\times n$ box of the hexagonal lattice. Assume that $n=(2k+1)\ell$; where $k,\ell>0$ are positive integers. We first divide the $n\times n$ box into $k^2$ rooms, each of which has size $\ell\times \ell$, as well as length $\ell$ corridors joining different rooms, or joining rooms with the boundary $\partial \Lambda_n$; see Figure \ref{fig:rc}. 
\begin{figure}[htbp]
\includegraphics[width=0.8\textwidth]{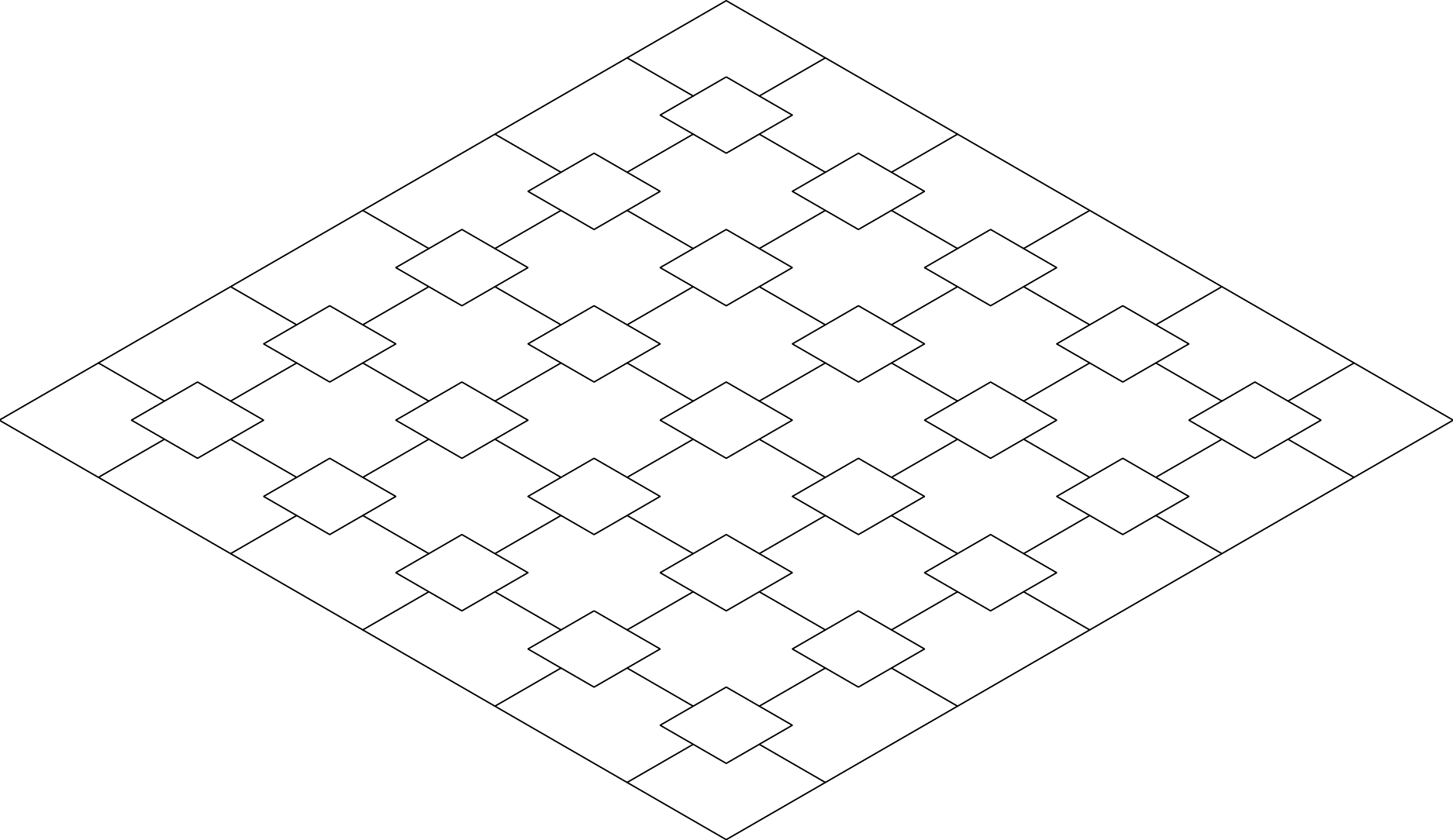}
   \caption{Rooms and Corridors. The largest rhombus represents the $n\times n$ box of $\Lambda_n$ of $\HH$; each small rhombus inside $\Lambda_n$ represents a $\ell\times \ell$ room; and corridors are represented by the line segments joining different rooms or joining a room to the boundary $\partial \Lambda_n$.}
   \label{fig:rc}
\end{figure}

For two boundary conditions $\tau$ and $\tau'$ different in at most 6 edges, we shall construct a coupling $(\alpha,\alpha')$, such that both $\alpha$ and $\alpha'$ are random 1-2 model configurations on $\Lambda_n$; $\alpha$ has the distribution $\mu_{\Lambda_n}^{\tau}$ and $\alpha'$ has the distribution $\mu_{\Lambda_n}^{\tau'}$, where $\mu_{\Lambda_n}^{\tau}$ (resp.\ $\mu_{\Lambda_n}^{\tau'}$) is the probability measure for 1-2 model configurations on $\Lambda_n$ conditional on the boundary condition $\tau$ (resp.\ $\tau'$), such that the probability of a configuration is proportional to the products of local weights at each vertex; and the local weights are given by $a,b,c>0$. Each time we update the configurations either in a room or in a corridor. Index all the rooms by positives integers $1,2,\ldots,k^2$. Two rooms are adjacent if they are joined by a corridor. Two rooms are *-adjacent if they are joined by a path consisting of two non-parallel corridors and a room in between. We first find the rooms with least distance to the edges where $\tau$ and $\tau'$ disagree, and then identify the one with least index among these rooms; denote the room identified by $R_1$. Sample $\alpha$ and $\alpha'$ in $R_1$ according to the distribution $\mu_{\Lambda_n}^{\tau}$ and $\mu_{\Lambda_n}^{\tau'}$, respectively. If $\alpha\neq\alpha'$ in $R_1$, mark the room as ``bad''. If $\alpha=\alpha'$ in this room, sample $\alpha$ and $\alpha'$ in the four corridors incident to $R_1$ conditional on $\tau\cup \alpha(R_1)$ and $\tau'\cup \alpha'(R_1)$, respectively. If $\alpha=\alpha'$ in all the four corridors, declare $R_1$ to be ``good''; otherwise declare $R_1$ to be ``bad''. 

 A room is called blank if it hasn't been declared to be ``good'' or ``bad''. We continue to update $\alpha$ and $\alpha'$ on blank rooms until there are no blank rooms *-adjacent to bad rooms. Otherwise we update $\alpha$ and $\alpha'$ in the room with least index and *-adjacent to a bad room by similar process as above, conditional on $\tau\cup\alpha_{s}$ and $\tau'\cup\alpha'_s$ respectively, where $\alpha_s$ and $\alpha_s'$ denote the configurations updated already up to now in $\alpha$ and $\alpha'$.   We then declare the room to be either ``good'' or ``bad'' according to the same criterion as above. 
 
 When there are no blank rooms *-adjacent to bad rooms, we update $\alpha$ and $\alpha'$ together in the remaining blank rooms according to the same distribution; in particular, we always assign $\alpha=\alpha'$ in all the remaining blank rooms and their incident corridors that haven't been explored yet. Finally we update $\alpha$ and $\alpha'$ on all the connected components of $\Lambda_n$ that haven't been explored yet conditional on $\tau\cup \alpha_s$ and $\tau'\cup\alpha'_s$, respectively.
 
 Let $\Delta$ be a finite subset of $\Lambda_n$. Let $\mu_{\Lambda_n,\Delta}^{\tau}$ (resp.\ $\mu_{\Lambda_n,\Delta}^{\tau'}$) be the probability measure of 1-2 model configurations on $\Delta$ conditional on the boundary condition $\tau$ (resp.\ $\tau'$) on $\partial \Lambda_n$. Then
 \begin{eqnarray*}
 \|\mu_{\Lambda_n,\Delta}^{\tau}-\mu_{\Lambda_n,\Delta}^{\tau'}\|_{TV}\leq \mathrm{Pr}(\alpha(\Delta)\neq \alpha'(\Delta)).
\end{eqnarray*}
where $\mathrm{Pr}$ denotes the probability of an event. 

A $*$-path is a sequence of rooms $r_1,r_2,\ldots,r_j$ such that for $1\leq i\leq j-1$, $r_i$ and $r_{i+1}$ are $*$-adjacent. A $**$-path is a sequence of rooms $s_1,\ldots,s_t$ such that for $1\leq i\leq t-1$, $s_i$ and $s_{i+1}$ are not $*$-adjacent but both are $*$-adjacent to a common room $u_i$.
By the coupling process above, if $\alpha(\Delta)\neq \alpha'(\Delta)$, then there exists a $*$-path consisting of bad rooms joining the edges where $\tau\neq \tau'$ to $\Delta$. We infer that there exists a $**$-path consisting of bad rooms joining the edges where $\tau\neq \tau'$ to $\Delta$.

The coupling process shows that the probability that a room is declare bad is at most
\begin{eqnarray*}
p_{b}&=&\max_{\tau,\tau'\in B(3\ell)}\|\mu_{\Lambda_{3\ell},\Lambda_{\ell}}^{\tau}-\mu_{\Lambda_{3\ell},\Lambda_{\ell}}^{\tau'}\|_{TV}+2\max^*_{\tau,\tau'\in B(\ell)}\|\mu_{\Lambda_n,M_1(\Lambda_{\ell})}^{\tau}-\mu_{\Lambda_{\ell}, M_1(\Lambda_n)}^{\tau'}\|_{TV}\\&&+2\max^{**}_{\tau,\tau'\in B(\ell)}\|\mu_{\Lambda_{\ell},M_2(\Lambda_{\ell})}^{\tau}-\mu_{\Lambda_{\ell}, M_2(\Lambda_{\ell})}^{\tau'}\|_{TV}
\end{eqnarray*}
where $\max^{*}$ and $\max^{**}$, $B(\ell)$ are defined as in Definition \ref{d62}. Then we have
\begin{eqnarray*}
&&\max_{\tau,\tau'\in B(n)} \|\mu_{\Lambda_n,\Delta}^{\tau}-\mu_{\Lambda_n,\Delta}^{\tau'}\|_{TV}\\
&\leq &\mathrm{Pr}(\mathrm{there\ exists\ a\ **-path\ from\ \tau\neq\tau'\ to \Delta})\\
&\leq &\sum_{\pi:|\pi|\geq C_1 d(\tau\neq \tau',\Delta)}p_b
\end{eqnarray*}
where the sum is over all self-avoiding $**$-path starting from $\tau\neq\tau'$ whose length is a least $C_1 d(\tau\neq \tau',\Delta)$; $C_1$ is a constant and $d(\tau\neq \tau',\Delta)$ is the graph distance between the edges where $\tau\neq \tau'$ and $\Delta$. Such self-avoiding $**$-path are exactly self-avoiding walks (SAWs) on the graph $G$ obtained from the square grid $\ZZ^2$ by joining each vertex $(p,q)$ with $(p,q\pm 2)$, $(p\pm 1,q\pm 2)$, $(p\pm 2,q\pm2)$, $(p\pm 2,q\pm1)$ and $(p\pm 2,q)$. The number of $n$-step SAWs starting from a fixed vertex is asymptotically $\nu^n$, when $n$ is sufficiently large, where $\nu$ is the connective constant. Then the Theorem follows.
See also \cite{DS85}.
\end{proof}

See the appendix for an investigation of efficient algorithms to check the condition $F(N,\nu^{-1})$. Note that if there exists a positive integer $N$, such that $F(N,\nu^{-1})$ holds, then the strong spatial mixing implies that there exists a unique infinite-volume Gibbs measure for the 1-2 model configurations on the hexagonal lattice. It is proved in Theorem 6.3 of \cite{GL1} that when $\sqrt{a}>\sqrt{b}+\sqrt{c}$, or $\sqrt{b}>\sqrt{a}+\sqrt{c}$, or $\sqrt{c}>\sqrt{a}+\sqrt{b}$, the infinite-volume Gibbs measures for 1-2 model configurations on $\HH$ are not unique. Therefore when $\sqrt{a}>\sqrt{b}+\sqrt{c}$, or $\sqrt{b}>\sqrt{a}+\sqrt{c}$, or $\sqrt{c}>\sqrt{a}+\sqrt{b}$, $F(N,\nu^{-1})$ never holds for any positive integer $N$.

\subsection{Block dynamics and comparison}\label{ssc:bd}

Recall that $\Lambda_n=(V_n,E_n)$ is an $n\times n$ box of $\HH$. Let $F_n$ be the set of faces of $\Lambda_n$. Assume the Assumption of Theorem \ref{t34} holds.
For $1\leq i\leq n+\ell-1$, $1\leq j\leq n+\ell-1$, define the $(i,j)$th block of $\Lambda_n$ to be the vertex set
\begin{eqnarray*}
V_{i,j}=\{(p,q):\max\{i-\ell+1,1\}\leq p\leq \min\{i,n\},\max\{j-\ell+1,1\}\leq q\leq \min\{j,n\}\}.
\end{eqnarray*}
Let $E_{i,j}$ be the set of edges whose both endpoints are in $V_{i,j}$.

Recall also that for an admissible boundary condition $b_{\Lambda_n}$, $\Omega_{\Lambda_n,b_{\Lambda_n}}$ is the set of all the 1-2 model configurations on $\Lambda_n$ with boundary condition $b_{\Lambda_n}$. Let $\Omega_{\Lambda_n,b_{\Lambda_n}}$ be the state space.
The block dynamics for the 1-2 model on $\Lambda_n$ is the Markov chain defined as follows. 
Let 
\begin{eqnarray*}
\Omega_{\sigma,E_{i,j}}=\{\tau\in\Omega_{\Lambda_n,b_{\Lambda_n}}:\tau(e)=\sigma(e),\qquad\mathrm{for\ all}\ e\notin E_{i,j}\},
\end{eqnarray*}
Let $\pi_{i,j}^{\sigma}$ be a probability measure on $\Omega_{\sigma,E_{i,j}}$ such that the probability of a configuration in $\Omega_{\sigma,E_{i,j}}$ is proportional to the product of local weights at each vertex, where the local weights of each vertex are given by parameters $a,b,c>0$.

We now describe the block dynamics Markov chain. At each step
a block $V_{i,j}$ is picked uniformly at random among the $(n-\ell+1)^2$ blocks, and the configuration $\sigma$ is updated according to the $\pi_{i,j}^{\sigma}$ on $\Omega_{\sigma, E_{i,j}}$. 
Define the transition matrix
\begin{eqnarray*}
P_{E_{i,j}}(\sigma,\tau)=\pi_{i,j}^{\sigma}(\tau).
\end{eqnarray*}
for $\sigma,\tau\in \Omega_{\Lambda_n,b_{\Lambda_n}}$. In particular, if $\tau\notin \Omega_{\sigma,E_{i,j}}$, then $P_{E_{i,j}}(\sigma,\tau)=0$. The block dynamics has the transition matrix
\begin{eqnarray}
\tilde{P}:=\frac{1}{(n+\ell-1)^2}\sum_{i=0}^{n+\ell-1}\sum_{j=0}^{n+\ell-1}P_{E_{i,j}}.\label{bdn}
\end{eqnarray}

Let $\pi$ be the probability measure on $\Omega_{\Lambda_n,b_{\Lambda_n}}$, such that the probability of a configuration is proportional to the product of local weights at each vertex, and the local weights are given by parameters $a,b,c>0$; see Figure \ref{fig:sign}.
Again the lemma below is straightforward. 
\begin{lemma}The block dynamics with state space $\Omega_{\Lambda_n,b_{\Lambda_n}}$ and transition matrix given  by (\ref{bdn}) is an irreducible, reversible, aperiodic Markov chain whose stationary distribution is the distribution $\pi$ on $\Omega_{\Lambda_n,b_{\Lambda_n}}$.
\end{lemma}

We shall apply the comparison theorem which relates the spectral gap of the block dynamics to that of the single-site dynamics.

Let $\sigma,\tau\in \Omega_{\Lambda_n,b_{\Lambda_n}}$, such that $\sigma$ and $\tau$ differ only in $E_{ij}$. Let $D_{\sigma}$ (resp.\ $D_{\tau}$) be the dimer configuration on $\HH_{\Delta,\Lambda_n}$ corresponding to $\sigma$ (resp.\ $\tau$). We have
\begin{eqnarray}
d(D_{\sigma},D_{\tau})\leq C_1 \ell^3,\label{dub}
\end{eqnarray}
where $C_1>0$ is a constant independent of $\ell$, and $d(D_{\sigma},D_{\tau})$ is defined by (\ref{d12}).

Let $\xi=(\sigma_0,\tau_0)$, where $\sigma_0$ and $\tau_0$ agree everywhere except at a single edge $e$ or a single face $f$, as defined in (2) or (3) of Section \ref{mch}. Let $Q(\xi)$, $\tilde{Q}(\xi)$ be defined as in (\ref{q}), (\ref{tq}), respectively.

For $\tau\in \Omega_{\sigma,E_{ij}}$ for some $(i,j)$, $\Gamma_{\sigma\tau}$ is a path consisting of single-edge or single-face movements from $\sigma$ to $\tau$, and $|\Gamma_{\sigma\tau}|$ is the number of such movements. We fix a path $\Gamma_{\sigma\tau}$ for each pair $(\sigma,\tau)$ such that $\tau\in\Omega_{\sigma,E_{i,j}}$ for some $(i,j)$. The path $\Gamma_{\sigma\tau}$ is obtained from the process as described in the proof of Lemma \ref{ir}. Then we have
\begin{eqnarray*}
|\Gamma_{\sigma\tau}|\leq d(D_{\sigma},D_{\tau})\leq C_1\ell^3,
\end{eqnarray*}
by (\ref{dub}).
Then
\begin{eqnarray*}
R_{\xi}:&=&\frac{1}{Q(\xi)}\sum_{\sigma,\tau:\xi\in \Gamma_{\sigma\tau}}\pi(\sigma)\tilde{P}(\sigma,\tau)|\Gamma_{\sigma\tau}|\\
&\leq& C_1\ell^3\sum_{\sigma,\tau:\xi\in \Gamma_{\sigma\tau}}\frac{1}{(n+\ell-1)^2}\sum_{i=1}^{n+\ell-1}\sum_{j=1}^{n+\ell-1}\frac{\tilde{P}_{E_{i,j}}(\sigma,\tau)\pi(\sigma)}{P(\sigma_0,\tau_0)\pi(\sigma_0)}
\end{eqnarray*}
Since $\sigma$ and $\sigma_0$ differ only on edges of $E_{ij}$, we have
\begin{eqnarray*}
\left(\frac{\min\{a,b,c\}}{\max\{a,b,c\}}\right)^{2\ell^2}\leq \frac{\pi(\sigma)}{\pi(\sigma_0)}\leq \left(\frac{\max\{a,b,c\}}{\min\{a,b,c\}}\right)^{2\ell^2}
\end{eqnarray*}
 Let
\begin{eqnarray}
M&:=&\max_{0\leq i\leq n-\ell, 0\leq j\leq n-\ell}|E_{i,j}|\leq C_2\ell^2\label{m}\\
M^*&:=&\max\{\max_{e\in E_n}|\{(i,j),e\in E_{i,j}\}|,\max_{f\in F_n}|\{(i,j):f\subseteq E_{i,j}\}|\}\leq C_2\ell^2\label{ms}
\end{eqnarray}
We have
\begin{eqnarray*}
\frac{\tilde{P}_{E_{i,j}}(\sigma,\tau)}{P(\sigma_0,\tau_0)}\leq \frac{2(|E_n|+|F_n|)[\max\{a,b,c\}]^2}{[\min\{a,b,c\}]^2}\pi_{i,j}^{\sigma}(\tau),
\end{eqnarray*}
and
\begin{eqnarray}
R_{\xi}&\leq&\frac{2C_1\ell^3(|E_n|+|F_n|)}{(n-\ell+1)^2}\left(\frac{\max\{a,b,c\}}{\min\{a,b,c\}}\right)^{2\ell^2+2}\sum_{\sigma,\tau:\xi\in \Gamma_{\sigma\tau}}\sum_{i=0}^{n-\ell}\sum_{j=1}^{n-\ell}\pi_{i,j}^{\sigma}(\tau)\notag\\
&\leq & C_3\ell^3 M^*\left(\frac{\max\{a,b,c\}}{\min\{a,b,c\}}\right)^{2\ell^2+2}
\leq  C_5\ell^5 \left(\frac{\max\{a,b,c\}}{\min\{a,b,c\}}\right)^{2\ell^2+2}\label{rn}
\end{eqnarray}
where the last inequality follows from (\ref{ms}).

We obtain the lemma below

\begin{lemma}Let $\gamma_B$ (\resp,\ $\gamma$) be the spectral gap for the block (\resp,\ single-edge or single face) dynamics with state space $\Omega_{\Lambda_n,b_{\Lambda_n}}$, then
\begin{eqnarray*}
\gamma_B\leq C_5\ell^5\left(\frac{\max\{a,b,c\}}{\min\{a,b,c\}}\right)^{2\ell^2+2}\gamma,
\end{eqnarray*} 
where $C_5>0$ is a  constants.
\end{lemma}

\begin{proof}The lemma follows from Theorem \ref{tm41} and (\ref{rn}).
\end{proof}

Let $\gamma_B^*$ (\resp,\ $\gamma^*$) be the absolute spectral gap for the block (\resp,\ single-edge or single face) dynamics with state space $\Omega_{\Lambda_n,b_{\Lambda_n}}$
 Let $t_{mix}$ (resp.\ $t_{rel}$) be the mixing time (resp.\ relaxation time) for the Markov chain with transition matrix defined by (\ref{tmt}) and state space $\Omega_{\Lambda_n,b_{\Lambda_n}}$.
Then by Lemma \ref{l52} and Theorem \ref{tm42}, we have
\begin{eqnarray*}
t_{mix}(\epsilon)\leq \log\left(\frac{|\Omega_{\Lambda_n,b_n}|}{\epsilon}\right)t_{rel}\leq  C_5\ell^5 4^{C_4\ell^2}\log\left(\frac{|\Omega_{\Lambda_n,b_n}|}{\epsilon}\right)\frac{1}{\gamma_B}.
\end{eqnarray*}
where we use
\begin{eqnarray*}
t_{rel}=\frac{1}{\gamma^*}=\frac{1}{\gamma},
\end{eqnarray*}
by Proposition \ref{ap}.

We now describe how to couple the block dynamics started from $\sigma$ with the block dynamics started from $\tau$, in the case that $\sigma,\tau\in \Omega_{\Lambda,b_{\Lambda}}$ and differ at only one edge $e$ or only one face $f$, as described in (2) and (3). Always select the same block to update the two chains. If a block is selected which contains the edge $e$ or the face $f$, then the two chains can be updated together, and the difference of $\sigma$ and $\tau$ can be eliminated. If a block is selected which contains no vertices of the edge $e$ and no vertices of the face $f$, then the two chains are also updated together, and after the update, the two new configurations still differ at only the edge $e$ or the face $f$. Two other cases need to be considered
\begin{Alist}
\item $e$ is incident to exactly one vertex in the selected block; or
\item at least one vertex in $f$ is in the selected block and at least one vertex in $f$ is outside the selected block.
\end{Alist}

We first consider case A. Assume that a block $V_{ij}$ is selected. Let $e'\in E_{ij}$ be an edge. By Theorem \ref{t34}, we have
\begin{eqnarray*}
\mathbb{P}(e'\ \mathrm{has\ different\ states\ in\ the\ update})\leq C_7 e^{-C_8\mathrm{dist}(e,e')}
\end{eqnarray*}
where $\partial V_{{i,j}}$ consists of all the edges with exactly one vertex in $V_{ij}$ and one vertex outside $V_{ij}$.
Then we can compute 
\begin{eqnarray*}
\mathbb{E}(\#\ \mathrm{of\ edges\ in}\ V_{ij}\ \mathrm{such\ that\ the\ two\ updates\ disagree})&\leq& C_7\sum_{e'\in E_{ij}}e^{-C_8\mathrm{dist}(e,e')}\leq C_9,\\
\end{eqnarray*}
where $C_9>0$ is a constant independent of $\ell$.

Let $\rho(\sigma,\tau)$ be the number of edges with different states in $\sigma$ and $\tau$. First let us treat the case when $\sigma$ and $\tau$ differ at exactly one edge $e$. Then $\rho(\sigma,\tau)=1$. Let $(X_1,Y_1)$ be the pair of configurations obtained after one step of coupling. Since $\ell^2$ of the $(n+\ell-1)^2$-blocks will contain the edge $e$, and at most $4\ell$ blocks have $e$ as an boundary edge, we have
\begin{eqnarray*}
\mathbb{E}_{\sigma,\tau}\rho(X_1,Y_1)\leq 1-\frac{\ell^2}{(n+\ell-1)^2}+\frac{4\ell C_9}{(n+\ell-1)^2}
\end{eqnarray*}
Choose $\ell=4C_9+1$, then
\begin{eqnarray}
\mathbb{E}_{\sigma,\tau}\rho(X_1,Y_1)\leq\left( 1-\frac{1}{n^2}\right)\rho(\sigma,\tau).\label{ctn}
\end{eqnarray}
If $\sigma,\tau$ differs at exactly one face $f$, similar arguments will also give us (\ref{ctn}).

By \ref{tm53}, we have
\begin{eqnarray*}
\gamma_B\geq \gamma_{B,*}\geq \frac{1}{n^2},
\end{eqnarray*}
where $\gamma_B$ (resp.\ $\gamma_{B,*}$) is the spectral gap (absolute spectral gap) of the block dynamics. Since
\begin{eqnarray*}
\log\left(\frac{|\Omega_{\Lambda_n,b_{\Lambda_n}}|}{\epsilon}\right)\leq C(\epsilon)n^2,
\end{eqnarray*}
By  (\ref{tm}), (\ref{s1}), (\ref{s2}), we have
\begin{eqnarray*}
t_{mix}\leq C n^4,
\end{eqnarray*}
for the Markov chain described by (\ref{tmt}). By Theorem \ref{tm44}, we have
\begin{eqnarray*}
t_{mix}\geq B n^2,
\end{eqnarray*}
for constants $C,B>0$ independent of $n$. This completes the proof of Theorem \ref{m2}.

Next we consider the mixing time of the block dynamics with transition matrix given by (\ref{bdn}). 

\begin{theorem}\label{tm67}Let $\tilde{P}$ defined by (\ref{bd}) be the transition matrix for the block dynamics of 1-2 model configurations on $\Omega_{\Lambda_{k,n},b_{k,n}}$, then
\begin{eqnarray*}
B' n^2 \leq \tilde{t}_{mix}\leq C' n^2\log n,
\end{eqnarray*}
where $B'>0$, $C'>0$ are positive constants independent of $n$, and $\tilde{t}_{mix}$ is defined by (\ref{mtm}) with respect to $\tilde{P}$.
\end{theorem}
\begin{proof}The lower bound $B' n^2$ is obtained following Theorem \ref{tm45}, and the upper bound is obtained following Theorem \ref{tm44}.
\end{proof}

\section{Appendix: How to check the condition $F(N,\nu^{-1})$}
By Theorem \ref{t34}, in order to prove the strong mixing condition for Gibbs measures of the 1-2 model on all the sufficiently large boxes with uniform constant, it suffices to show that $F(N,\nu^{-1})$ holds for some positive integer $N$. In this section, we investigate efficient algorithms to check the condition $F(N,\nu^{-1})$.

By the measure-preserving correspondence of 1-2 model configurations on $\HH$ and dimer configurations on $\HH_{\Delta}$, as well as the domain Markov property, for the boundary condition of an $n\times n$ box $\Lambda_n$ of $\HH$, it suffices to consider dimer configurations in all the hexagons crossing $\partial\Lambda_n$. (For the $3\times 3$ box $\Lambda_3$, $\partial\Lambda_3$ is given by the blue lines in Figure \ref{fig:cl3}).

Each dimer configuration on $\HH_{\Delta}$ satisfies the following two constraints
\begin{itemize}
\item each vertex of $\HH$ has exactly one incident preset bisector edge in $\HH_{\Delta}$; and
\item around each hexagon of $\HH$, there are an even number of present bisector edges.
\end{itemize}

Indeed, for any configuration on bisector edges satisfying the above two constraints, the configuration can be uniquely extended to a dimer configuration on $\HH_{\Delta}$. As a result, the influence of boundary conditions to dimer configurations in $\Lambda_n$ depends only on the parity of the number of present bisector edges outside $\Lambda_n$ in each hexagon crossing $\partial\Lambda_n$. Let $n\geq 3$. All the hexagons crossed by $\partial\Lambda_n$ can be classified into 3 different types:
\begin{enumerate}
\item The hexagon has 5 vertices outside $\Lambda_n$, and 1 vertex inside $\Lambda_n$. The two hexagons on the left and right corners of Figure \ref{fig:cl3} are of this type.
\item The hexagon has 4 vertices outside $\Lambda_n$, and 2 vertices inside $\Lambda_n$. The two hexagons on the top and bottom corners of Figure \ref{fig:cl3} are of this type.
\item The hexagon has 3 vertices outside $\Lambda_n$ and 2 vertices inside $\Lambda_n$. All the hexagons crossing $\partial\Lambda_n$ but not on the corners are of this type.
\end{enumerate}

We say a hexagon $h$ crossing $\partial \Lambda_n$ is \textbf{positive} (resp.\ \textbf{negative}) with respect to a boundary condition $\tau$, if in $h$, an even (resp.\ odd) number of incident bisector edges to vertices outside $\Lambda_n$ are present in $\tau$. Note that for an admissible boundary condition, there are an even number of hexagons crossing $\partial\Lambda_n$ that are negative with respect to the boundary condition.

We will treat different types of hexagons differently.

If a Type (1) hexagon $h_1$ is positive (resp. negative) with respect to the boundary condition, let $e$ be the bisector edge in $h_1$ incident to the unique vertex $v$ of $h_1$ in $\Lambda_n$; then $e$ must be absent (resp.\ present), in which case we remove $e$ (resp.\ all the incident edges of $v$), and treat the two adjacent hexagons of $h_1$ crossing $\partial\Lambda_n$ as Type-3 hexagons (Type-2 hexagons).

If a Type (2) hexagon $h_1$ is positive (resp.\ negative) with respect to the boundary condition, then we replace the graph $\HH_{\Delta}$ in the hexagon by an edge (resp.\ a vertex and two edges), see the left graph (resp.\ the right graph) of Figure \ref{fig:g2}.

\begin{figure}[htbp]
\includegraphics[width=0.2\textwidth]{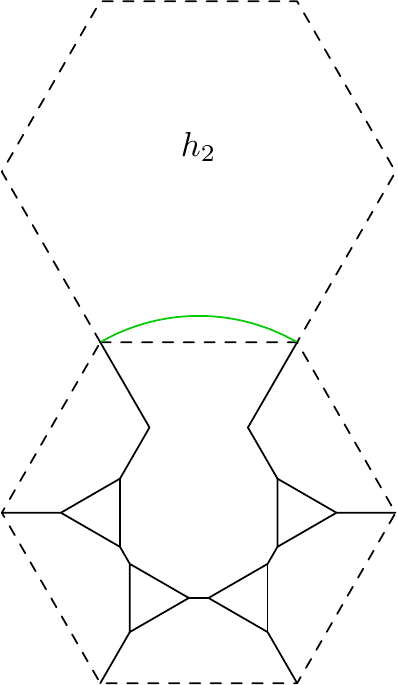}\qquad\qquad\qquad\qquad\includegraphics[width=0.2\textwidth]{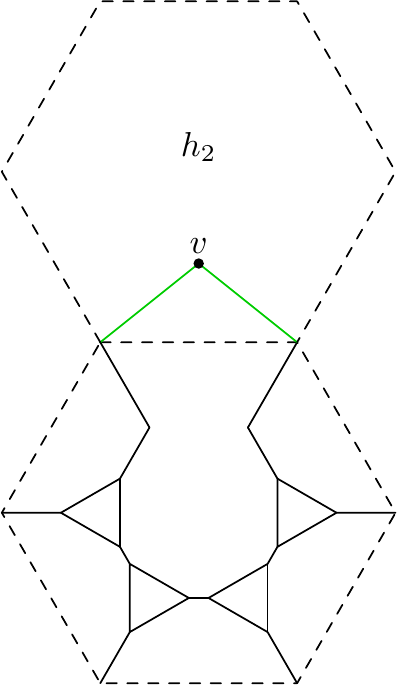}
\caption{Replace a positive (resp.\ negative) Type (2) hexagon $h_2$ on the boundary by a gadget represented by green lines on the left (resp.\ right) graph.}
\label{fig:g2}
\end{figure}

If a Type (3) hexagon $h_3$ is positive (resp.\ negative) with respect to the boundary condition, then we replace the graph $\HH_{\Delta}$ in the hexagon by a gadget in the left graph (resp.\ the right graph) Figure \ref{fig:g3}.

\begin{figure}[htbp]
\includegraphics[width=0.35\textwidth]{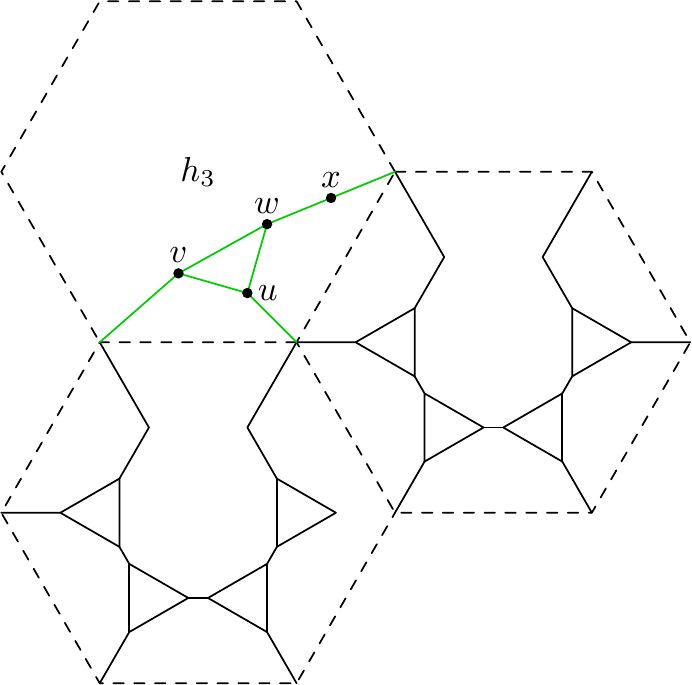}\qquad\qquad\qquad\qquad\includegraphics[width=0.35\textwidth]{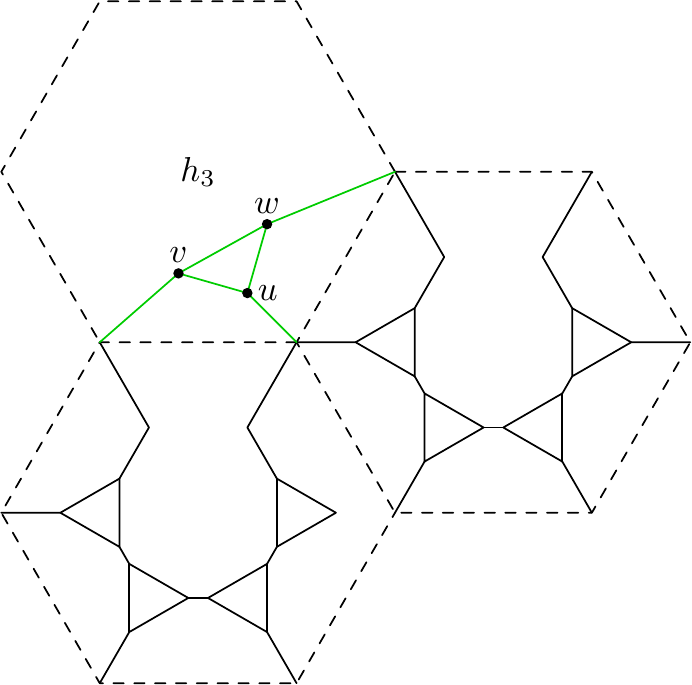}
\caption{Replace a positive (resp.\ negative) Type (3) hexagon $h_3$ on the boundary by a gadget represented by green lines on the left (resp.\ right) graph.}
\label{fig:g3}
\end{figure}

Therefore, to consider the Gibbs measures for dimer configurations on $\Lambda_n$ with different boundary conditions, it suffices to consider Gibbs measures on different finite graphs. After given a clockwise odd orientation (which always exists on a planar graph), these can be computed by computing determinants and Pfaffians.

\bigskip
\noindent\textbf{Acknowledgements.} The author thanks Richard Kenyon for suggesting the problem solved in this paper, and Yuval Peres for suggesting the path method and comparison with block dynamics.
The author acknowledges support from National Science Foundation under grant 1608896.

\bibliography{rr12mc}
\bibliographystyle{amsplain}

\end{document}